\documentclass[1 [leqno,11pt]{amsart}
\usepackage{amssymb, amsmath}

\let\pa=\partial
\let\f=\frac
\let\wt=\widetilde
\let\wh=\widehat
\let\D=\Delta
\let\om=\omega
\let\lam=\lambda
\def\v{{\rm v}}
\def\s{\sigma}
\def\d{\delta}

\def\cA{{\mathcal A}}
\def\cB{{\mathcal B}}
\def\cC{{\mathcal C}}

\def\cE{{\mathcal E}}
\def\cF{{\mathcal F}}

\def\cS{{\mathcal S}}

\def\eqdef{\buildrel\hbox{\footnotesize def}\over =}
\def\Z{\mathop{\mathbb Z\kern 0pt}\nolimits}
\def\N{\mathop{\mathbb N\kern 0pt}\nolimits}
\def\Q{\mathop{\mathbb Q\kern 0pt}\nolimits}
\def\R{\mathop{\mathbb R\kern 0pt}\nolimits}
\def\Supp{\mathop{\rm Supp}\nolimits\ }

\def\dive{{\mathop{\rm div}\nolimits}\,}
\def\diveh{{\mathop{\rm div}_{\rm h}\nolimits}\,}

\def\curlh{{\mathop{\rm curl}_{\rm h}\nolimits}\,}
\def\nablah{\nabla_{\rm h}}

\def\uh{u^{\rm h}}
\def\uhcurl{u^{\rm h}_{\mathop{\rm curl}\nolimits}}

\def\vh{v^{\rm h}}

\def\h{{\rm h}}

\def\Deltah{\Delta_{\rm h}}

\def\baru{\bar u}
\def\baruh{\baru^\h}
\def\wtu{\wt u}
\def\barom{\bar \omega}

\def\wtom{\wt \omega}
\def\pD{\pa_3\D^{-1}}
\def\na{\nabla}

\def\dhk{\Delta_k^{\rm h}}

\def\dhkp{\Delta_{k'}^{\rm h}}

\def\LVLH {L^\infty_{\rm v} (L^2_{\rm h})}

\def\hh{{H^{\frac12,0}}}%H horizontal half
\def\hm{{H^{-\frac12,0}}}
\def\hmo{{H^{-\frac12,1}}}
\def\bh{{B^{0,\f12}_{2,1}}}
\def\boh{{B^{1,\f12}_{2,1}}}
\def\nvn{\|\nabla v^3\|}%nabla v norm
\def\nomn{\|\nabla \omega\|}%nabla omega norm

\def\la{\lambda}
\def\e{\epsilon}
\def\ve{\varepsilon}

\def\eqdefa{\buildrel\hbox{\footnotesize def}\over =}

\newcommand{\Rmnum}[1]{\uppercase\expandafter{\romannumeral #1} }

\newcommand{\beq}{\begin{equation}}
\newcommand{\eeq}{\end{equation}}
\newcommand{\ben}{\begin{eqnarray}}
\newcommand{\een}{\end{eqnarray}}
\newcommand{\beno}{\begin{eqnarray*}}
\newcommand{\eeno}{\end{eqnarray*}}

 \numberwithin{equation}{section}
\newcommand{\andf}{\quad\hbox{and}\quad}
\newcommand{\with}{\quad\hbox{with}\quad}

\newtheorem{defi}{Definition}[section]
\newtheorem{thm}{Theorem}[section]
\newtheorem{lem}{Lemma}[section]
\newtheorem{rmk}{Remark}[section]

\newtheorem{prop}{Proposition}[section]

\begin{document}

\title[Global solutions of $3$-D Navier-Stokes system ]
{Global solutions of $3$-D Navier-Stokes system with small unidirectional
derivative}

\author[Y. Liu]{Yanlin Liu}
\address[Y. Liu]{Academy of Mathematics $\&$ Systems Science
and   Hua Loo-Keng Center for Mathematical Sciences,
Chinese Academy of Sciences, Beijing 100190, CHINA.}
 \email{liuyanlin@amss.ac.cn}

\author[P. Zhang]{Ping Zhang}
\address[P. Zhang]{Academy of Mathematics $\&$ Systems Science
and  Hua Loo-Keng Key Laboratory of Mathematics, Chinese Academy of
Sciences, Beijing 100190, CHINA, and School of Mathematical Sciences,
University of Chinese Academy of Sciences, Beijing 100049, China.} \email{zp@amss.ac.cn}

\date{\today}

\begin{abstract}  Given initial data $u_0=(u_0^\h,u_0^3)\in H^{-\d,0}\cap H^{\f12}(\R^3)$ with both~$\uh_0$ and~$\nabla_{\rm h}\uh_0$ belonging to
~$L^2(\R^3)\cap L^\infty(\R_\v;L^2(\R^2_\h))$ and $u_0^\h\in L^\infty(\R_\v; H^{-\d}(\R^2_\h))$ for some $\delta\in ]0,1[,$ if  in addition
 $\pa_3u_0$ belongs to $H^{-\frac12,0}\cap H^{\frac12,0}(\R^3),$
 we prove that the classical $3$-D Navier-Stokes system has a unique global Fujita-Kato solution  provided that $\|\pa_3u_0\|_{H^{-\f12,0}}$ is
sufficiently small compared to a constant which  depends only on the norms of the initial data.
In particular, this result  provides some classes of large initial data
which  generate unique global solutions to  3-D Navier-Stokes system.
\end{abstract}

\maketitle



\setcounter{equation}{0}

\section{Introduction}
In this paper, we consider
the following $3$-D incompressible Navier-Stokes system:
\begin{equation*}
(NS)\quad \left\{\begin{array}{l}
\displaystyle \pa_t u +u\cdot\nabla u-\Delta u=-\nabla p, \qquad (t,x)\in\R^+\times\R^3, \\
\displaystyle \dive u = 0, \\
\displaystyle  u|_{t=0}=u_0,
\end{array}\right.
\end{equation*}
where $u=(u^1,u^2, u^3)$ stands for the  fluid  velocity and
$p$ for the scalar pressure function, which guarantees the divergence free condition of the velocity field.

\smallskip

In 1933,
  Leray   proved in the seminar paper \cite{lerayns}
that given a   finite energy initial data, $u_0,$ $(NS)$ has a global in time weak
solution $u$ which verifies the energy inequality
\beq
\label {energydissp}
\frac  12 \|u(t)\|_{L^2}^2 +\int_0^t \|\nabla u(t')\|_{L^2}^2 dt' \leq \frac 12 \|u_0\|_{L^2}^2\,.
\eeq
However, the regularity and uniqueness of such solutions is still one of the biggest open questions
in the field of mathematical fluid mechanics except the case when the initial data have special structure.
 For
instance,  with axi-symmetric initial velocity and without swirl component, Ladyzhenskaya \cite{La}  and independently Ukhovskii and Yudovich
\cite{UY}  proved the
existence of weak solution along with the uniqueness and regularity of such solution to $(NS)$.

 Fujita-Kato  \cite{fujitakato}  constructed local in time unique solution
 to $(NS)$ with initial data in $H^{\frac 12}(\R^3)$
\footnote{Through out this paper, we always designate $H^s$ to be
 homogeneous Sobolev spaces, and $B^s_{p,r}$ to be  homogeneous Besov spaces.}.
 Furthermore,  if  $\|u_0\|_{H^{\frac 12}}$ is sufficiently small, then the  solution exists globally in time (see \cite{kato} for similar result with initial data  in  $L^3(\R^3)$).
This result was extended by Cannone, Meyer and Planchon \cite{cannonemeyerplanchon} for initial data belonging to $
 B^{-1+\frac 3p}_{p,\infty}(\R^3)$ with $p\in ]3,\infty[.$
 The end-point result in this direction is given by Koch and Tataru  \cite{kochtataru}. They proved that given initial data being
sufficiently small in $\text{BMO}^{-1}(\R^3)$,  then $(NS)$ has a unique global  solution.
 We remark that for $p\in ]3,\infty[$, there holds
 $$H^{\f12}(\R^3) \hookrightarrow L^3(\R^3) \hookrightarrow B^{-1+\frac 3p}_{p,\infty}(\R^3) \hookrightarrow \text{BMO}^{-1}(\R^3)
 \hookrightarrow B^{-1}_{\infty,\infty}(\R^3),$$
 and the norms to the above spaces   are sclaing-invariant
 under the following transformation:
 \beq \label{S1eq2} u_\lambda(t,x)=\lambda u(\lambda^2 t,\lambda x) \andf u_{0,\lambda}(x)=\lambda u_0(\lambda x).\eeq
 We notice that for any solution $u$ of $(NS)$ on $[0,T],$ $u_\la$ determined by \eqref{S1eq2} is also a solution of $(NS)$ on $[0,T/\la^2].$ We remark that the largest
 space, which belongs to $\cS'(\R^3)$ and the norm of which is scaling invariant under \eqref{S1eq2}, is $B^{-1}_{\infty,\infty}(\R^3)$ (see \cite{meyer}).
Moreover, Bourgain and Pavlovi\'c \cite{BP08} proved that $(NS)$ is actually
ill-posed  with initial data in  $B^{-1}_{\infty,\infty}(\R^3).$
That is the reason why we call such kind of initial data, which generates a unique global solution to $(NS)$ and the $B^{-1}_{\infty,\infty}(\R^3)$
 norm of which
is large enough, as large initial data.

\smallskip

We now list some examples of large initial data which generate unique global
solutions to $(NS)$. First of all, for any initial data,
Raugel and Sell \cite{raugelsell} obtained the global well-posedness of $(NS)$ in a thin enough domain. This result was extended
by Raguel, Sell and Iftimie in \cite{iftimieraugelsell}  that $(NS)$ has a unique global  periodic solution provided that
the initial data $u_0$ can be split as $u_0=v_0+w_0$, with $v_0$ being a bi-dimensional solenoidal vector field in $L^2({\Bbb{T}}^2_\h)$ and $w_0\in H^{\frac 12}(\Bbb{T}^3)$, such that
$$\|w_0\|_{H^{\frac 12}(\Bbb{T}^3)}\exp \bigl(\|u_0\|_{L^2(\Bbb{T}^2_\h)}^2\bigr)\quad\mbox{is sufficiently small}.$$

Chemin and Gallagher \cite{cgens} constructed another class of examples of initial
data which is big in $B^{-1}_{\infty,\infty}(\Bbb{T}^3)$
and strongly oscillatory in one direction. More precisely, for any given positive integer
$N_0$, there exists some $N_1$ such that for any integer $N>N_1$, $(NS)$ has a unique global solution with initial data
$$u_0^N=\left(Nv^\h_0(x_h)\cos(Nx_3),-\dive_\h v_0^h(x_h)\sin(Nx_3)\right),$$
where $v_0^\h$ is any bi-dimensional solenoidal  field
 on $\Bbb{T}^2$ with
$$\Supp \hat v^\h_0\subset[-N_0,N_0]^2\quad\mbox{and}\quad
\|v_0^\h\|_{L^2(\Bbb{T}^2)}\leq C(\ln N)^{\frac 19}.$$

The other  interesting class of large initial data
which can generate unique global  solutions to $(NS)$
is the so-called
slowly varying data, which was introduced by Chemin and Gallagher
in \cite{CG10} (see also \cite{CGZ, CZ15}).

\begin{thm}[Theorem 3 of \cite{CG10}]\label{thmCG}
{\sl Let $\vh_0=(v^1_0,v^2_0)$ be a horizontal, smooth divergence free vector field
on $\R^3$ such that it and all its derivatives belong to
$L^2(\R^3)\cap L^2(\R_{x_3};H^{-1}(\R^2_\h))$;
let $w_0$ be a smooth divergence free vector field on $\R^3$. Then there exists a positive constant $\varepsilon_0$ such that if $\varepsilon\leq\varepsilon_0$,
then the initial data
\begin{equation}\label{slowinitialdata}
u_0^\varepsilon(x)=(\vh_0 +\varepsilon w^\h_0,w^3_0)(x_\h,\varepsilon x_3)
\end{equation}
generates a unique, global smooth solution $u^\varepsilon$ of $(NS)$.
}\end{thm}

On the other hand, Kukavica and  Ziane proved in \cite{KZ2} that if
$$\int_0^{T^\star} \|\partial_3 u(t,\cdot)\|^q_{L^p}dt<\infty \with \f2q+\f3p=2 \andf p\in \left]9/4, 3\right[, $$
 then local smooth solution of $(NS)$ can be extended beyond time $T^\ast.$

 Motivated by \cite{CG10} and \cite{KZ2}, we are going to study the global well-posedness of $(NS)$
 with initial data $u_0$ satisfying $\pa_3u_0$ being sufficiently small in some critical spaces. Before we present the main result, let us
 recall the definition of anisotropic Sobolev space $H^{s,s'}(\R^3)$.

\begin{defi}\label{defanisob}
{\sl For $(s,s')$ in $\R^2$, $H^{s,s'}(\R^3)$ denotes
the space of homogeneous tempered distribution~$a$  such~that
 $$
\|a\|^2_{H^{s,s'}} \eqdefa \int_{\R^3} |\xi_{\rm
h}|^{2s}|\xi_3|^{2s'} |\wh a (\xi)|^2d\xi <\infty\with \xi_{\rm
h}=(\xi_1,\xi_2)\andf \xi=(\xi_\h,\xi_3) .
 $$
}
\end{defi}

The main result of this paper states as follows.
\begin{thm}\label{thmmain}
{\sl  Let $\delta\in]0,1[,$ $u_0=(u_0^\h,u_0^3)\in H^{-\d,0}\cap H^{\f12}(\R^3)$ with both~$\uh_0$ and~$\nabla_{\rm h}\uh_0$ belonging to
~$L^2(\R^3)\cap L^\infty(\R_\v;L^2(\R^2_\h))$ and $u_0^\h\in L^\infty(\R_\v; H^{-\d}(\R^2_\h)).$ If we assume in addition that
$\pa_3u_0\in H^{-\frac12,0}\cap H^{\frac12,0}(\R^3),$ then
 there exists a  small enough positive constant $\e_0$ such that if
\begin{equation}\label{S1eq1}
\|\pa_3u_0\|_{H^{-\f12,0}}^2\exp\left(C\bigl(A_\d(\uh_0)+B_\d(u_0)\bigr)\right)
\leq\e_0,
\end{equation}
where
\beq \label{S1eq3} A_\d(u_0^\h)\eqdefa \biggl(\frac{\|\nabla _{\rm h}\uh_0\|_{\LVLH}^2
\|\uh_0\|_{L^\infty_{\rm v} (B^{-\d}_{2,\infty})_\h}
^{\frac 2 \delta}} {\|\uh_0\|_{\LVLH}^{\frac 2 \d}}
+\|\uh_0\|_{\LVLH}^{2}\biggr)
\exp\Big(C_\d(1+\|\uh_0\|_{\LVLH}^4)\Bigr),\eeq
and
\beq \label{S1eq4} B_\d(u_0)\eqdefa\|u_0\|_{H^{-\d,0}}^{\f12}\|u_0\|_{H^{\d,0}}^{\f12}
\|\pa_3u_0\|_{H^{-\frac12,0}}^{\f12}
\|\pa_3u_0\|_{H^{\frac12,0}}^{\f12}
\exp\bigl(CA_\d(\uh_0)\bigr),\eeq
$(NS)$ has a unique global solution
$u\in C\bigl(\R^+;H ^{\frac12}\bigr)\cap
L^2\bigl(\R^+;H ^{\frac 32}\bigr)$.
}\end{thm}

Let us present some examples of initial data the norm of  which are big in
$B^{-1}_{\infty,\infty}(\R^3)$, yet they satisfy the smallness condition \eqref{S1eq1}.

\begin{itemize}

\item[(1)] The first class of example is the slow variable data
given by \eqref{slowinitialdata}.  It is easy to observe that for any $\varepsilon\in]0,1]$, both
$A_\d(u_0^{\ve,\h})$ and $B_\d(u_0^\ve)$ have uniform upper bounds which are
independent of   $\ve.$  Whereas it follows from trivial calculation that
$$\|\pa_3u_0^\ve\|_{H^{-\f12,0}}\leq C\ve^{\f12}\bigl(
\|\pa_3v_0^\h\|_{H^{-\f12,0}}+\ve\|\pa_3w_0^\h\|_{H^{-\f12,0}}
+\|\pa_3w_0^3\|_{H^{-\f12,0}}\bigr).$$
So that as long as $\ve$ is sufficiently small,  the slow variable data
\eqref{slowinitialdata} satisfies \eqref{S1eq1}.
Hence Theorem \ref{thmCG} is  a  direct corollary of Theorem \ref{thmmain}.

\item[(2)] Motivated by \eqref{slowinitialdata}, for any positive integer $N$ and $(v_k,w_k), k=1,\cdots, N,$ satisfying the assumptions of Theorem \ref{thmCG},
we construct  initial  data which
is a linear combination of $N$ slowly varying parts as follows
\begin{equation}\label{ex2}
u_0^{(\varepsilon_1,\cdots,\ve_N)}(x)
=\sum_{k=1}^N(\vh_k+\varepsilon_k w^\h_k,w^3_k)(x_\h,\varepsilon_k x_3).
\end{equation}
It is easy to check that  $u_0^{(\varepsilon_1,\cdots,\ve_N)}$ satisfies
the smallness condition \eqref{S1eq1} provided
$\varepsilon_1,\cdots,\ve_N$ are sufficiently small.

\item[(3)] We can also consider the initial data which is slowly varying  in
one variable but fast varying in another variable. For instance,
\begin{equation}\label{ex3}
u_0^{(\ve,\lambda)}(x)
=(v_0^1+\ve w^1_0,\lam v_0^2+\lam\ve w_0^2,\lam w^3_0)(\lam x_1,x_2,\ve x_3).
\end{equation}
Then it follows from Sobolev inequality that
\begin{align*}
\|\pa_3u_0^{(\ve,\lambda)}\|_{H^{-\f12,0}}\leq & C\|\pa_3u_0^{(\ve,\lambda)}\|_{L^{\f43}_\h(L^2_\v)}\\
\leq &C\lam^{\f14}\ve^{\f12}\bigl\|\bigl(\pa_3v_0^\h,\pa_3w_0\bigr)\bigr\|_{L^{\f43}_\h(L^2_\v)}.
\end{align*}
On the other hand, we have
\begin{align*}
&A_\d(u_0^{(\ve,\lambda),\h})\le C\exp(C\lam^2),
\andf B_\d(u_0^{(\ve,\lambda)})\le C\exp\bigl(C\exp(C\lam^2)\bigr).
\end{align*}
Therefore as long as $\ve$ and $\lam$ verify the condition that $$\ve^{\f12}\exp\bigl(C\exp\bigl(\exp(\lam^2)\bigr)\bigr)$$
is still sufficiently small, $u_0^{(\ve,\lambda)}(x)$  satisfies the smallness condition \eqref{S1eq1}.

\item[(4)] The other class of examples can be obtained by cutting the horizontal frequency
of the  initial data. For instance, given smooth solenoidal vector field $u_0$ with
$$\Supp\widehat{u_0}\subset \bigl\{\xi=(\xi_\h,\xi_3)\in\R^3\,:\,|\xi_\h|>R\bigr\}$$
for some positive constant $R,$  we find
$$\|\pa_3u_0\|_{H^{-\f12,0}}\thicksim R^{-\f12}.$$
Thus $u_0$ satisfies the smallness condition \eqref{S1eq1}  provided
$R$ is sufficiently large.
\end{itemize}

We refer  Proposition 1.1
of \cite{CG10} for the calculation of  the
 $B^{-1}_{\infty,\infty}(\R^3)$ norm to the data given by \eqref{slowinitialdata},~
\eqref{ex2} and \eqref{ex3}.

\smallskip

Let us end this section with the notations we shall use in this context.

\noindent{\bf Notations:}  We designate  the $L^2$ inner product of $f$ and $g$  by  $\left(f | g\right)_{L^2}.$ We shall always denote $C$ to be a uniform constant
which may vary from line to line.
 $a\lesssim b$ means that $a\leq Cb$.
 $\left(d_j\right)_{j\in\Z}$ (resp.  $\left(d_{k,\ell}\right)_{(k,\ell)\in\Z^2}$ and $\left(c_j\right)_{j\in\Z}$)  stands for a generic element on the unit
sphere of $\ell^1(\Z)$  (resp. $\ell^1(\Z^2)$ and $\ell^2(\Z)$) so that $\sum_{j\in\Z}d_j=1$ (resp.  $\sum_{(k,\ell)\in\Z^2}d_{k,\ell}=1$ and
$\sum_{j\in\Z}c_j^2=1$).

\medskip

\section{Ideas and structure of the proof}\label{sec2}

In this section, we shall sketch the main ideas of the proof
  to Theorem \ref{thmmain}.
Given initial data $u_0$ satisfying the assumptions of Theorem \ref{thmmain}, classical Fujita-Kato theorem (\cite{fujitakato}) ensures that
$(NS)$ has a unique  solution
\beq \label{FTS} u\in \cE_{T^\ast}\eqdef
C\bigl([0,T^*[;H ^{\frac12}(\R^3)\bigr)\cap
L^2_{\rm loc}\bigl([0,T^*[;H^{\frac 32}(\R^3)\bigr),\eeq
where $T^\ast$ is the maximal existence time of this solution.
The goal of this paper
is to prove that $T^\ast=\infty$ under the smallness condition \eqref{S1eq1}.

   We first remark that the key ingredient used in \cite{CG10, CGZ, CZ15} is that
with a slow variable for the solution of $(NS)$, one can decompose  it as a sum of a large two-dimensional solution of
$(NS)$ with a parameter and a small three-dimensional  one.
 Here the assumption in \eqref{S1eq1}
motivates us to expect that $\pa_3u(t,x)$ should be small in some sense and therefore
the convection term, $u\cdot \na u,$ in $(NS)$ can be approximated by $u^\h\cdot\na_\h u,$  that is, under the smallness
condition \eqref{S1eq1}, we may approximate $(NS)$ by
 \beq \pa_t \baruh +\baruh\cdot\nablah\baruh
-\Delta \baruh=-\nablah \bar p,\quad\diveh \baruh = 0. \label{S2eq4} \eeq
However, we can not simply set the initial data $\baruh_0$ to be $\uh_0$,
which does not satisfy the horizontal divergence free condition.
To overcome this difficulty, let us recall the Biot-Sarvart's law for
 a $2$-D vector field $\uh:$
\begin{equation}\label{Helmholtz}
u^\h=u^\h_{\mathop{\rm curl}\nolimits}
+u^\h_{\mathop{\rm div}\nolimits},\quad \mbox{where}\quad
u^\h_{\mathop{\rm curl}\nolimits}\eqdef\nablah^\perp\Deltah^{-1}(\curlh u^\h)
,~u^\h_{\mathop{\rm div}\nolimits}\eqdef\nablah\Deltah^{-1}(\diveh\uh).
\end{equation}
Here and in the sequel, we always denote
$\curlh u^\h\eqdef\pa_1 u^2-\pa_2 u^1$, $\diveh u^\h\eqdef\pa_1 u^1+\pa_2 u^2,$
and $\nablah\eqdef(\pa_1,\pa_2),~\nablah^\perp\eqdef(-\pa_2,\pa_1)$,
$\Deltah\eqdef\pa_1^2+\pa_2^2$.
It is easy to observe that $\diveh u^\h_{\rm curl}=0$
and $\curlh u^\h_{\rm div}=0$.
This motivates us to  define $(\baru^\h,\bar p)$ via
\begin{equation}\label{eqtbaru}
\left\{\begin{array}{l}
\displaystyle \pa_t \baruh +\baruh\cdot\nablah\baruh
-\Delta \baruh=-\nablah \bar p,\\
\displaystyle \diveh \baruh = 0, \\
\displaystyle \baruh|_{t=0}=\baruh_0=\uhcurl|_{t=0}
=\nablah^\perp\Deltah^{-1}(\curlh\uh_0).
\end{array}\right.
\end{equation}
\eqref{eqtbaru} seemingly looks like  $(NS2.5D)$ in \cite{CZ15},
whose solution has been used to approximate the true solution of
$3$-D Navier-Stokes equations with a slow variable.
We remark that the system \eqref{eqtbaru} here is simpler due to the fact that there is no small parameter in the front of $\pa_3^2$ in \eqref{eqtbaru}.
For this system, we deduce from the proof of Theorem 1.2 of \cite{CZ15} that:

\begin{thm}
\label{S4thm1} {\sl Let~$\uh_0$ and~$\nabla_{\rm h}\uh_0$ be
in~$L^2(\R^3)\cap L^\infty(\R_\v;L^2(\R^2_\h))$. Then
~$\bar{u}^\h_0$ generates a unique global solution to~\eqref{eqtbaru}
in the space~$L^\infty (\R^+; L^2(\R^3)\cap L^\infty(\R_\v; H^1(\R^2_\h)))\cap L^2(\R^+; H^1(\R^3))$.

Moreover, if in addition~$\uh_0$ belongs to~$L^\infty(\R_\v;
H^{-\delta} (\R^2_\h))$ for some ~$\d$ in~$]0,1[$, then we have
\begin{equation}\begin{split}\label{S4eq1}
\int_0^{\infty} \|\nabla_{\rm h } \uh (t)\|_{L^\infty_{\rm v}(L^2_\h)}^2\,dt \leq & A_\d (\uh_0),
\end{split}\end{equation} where $A_\d (\uh_0)$ is determined by \eqref{S1eq3}.
}\end{thm}

Concerning the system \eqref{eqtbaru}, we have the following proposition:

\begin{prop}\label{propbaru}
{\sl Under the assumptions of Theorem \ref{S4thm1},  if in addition  both $\uh_0$ and $\pa_3\uh_0$
belong to $\cB^{\sigma_1,0}$ for some $\sigma_1\in ]-1,1[,$ \eqref{eqtbaru} has a unique global solution, \beq \label{ping1}
\baruh\in \frak{E}\eqdefa
L^\infty \left(\R^+; L^2(\R^3)\cap L^\infty(\R_\v; H^1(\R^2_\h))\right)\cap L^2\left(\R^+; H^1(\R^3)\cap L^\infty(\R_\v; H^1(\R^2_\h))\right) \eeq so that
 for any $t>0$ and for any $\sigma_2\in ]0,1[,$  we have
\begin{equation}\label{estimatebaru1}
\|\baruh\|_{{L}^\infty_t(B^{\sigma_1,\sigma_2}_{2,1})}
+\|\nabla\baruh\|_{{L}^2_t(B^{\sigma_1,\sigma_2}_{2,1})}
\leq C\|\uh_0\|_{\cB^{\sigma_1,0}}^{1-\sigma_2}
\|\pa_3\uh_0\|_{\cB^{\sigma_1,0}}^{\sigma_2}
\exp\bigl(CA_\d(\uh_0)\bigr),
\end{equation}
and
\begin{equation}\label{estimatebaru2}
\|\pa_3\baruh\|_{{L}^\infty_t(H^{\sigma_1,0})}
+\|\nabla\pa_3\baruh\|_{{L}^2_t(H^{\sigma_1,0})}
\leq C\|\pa_3\uh_0\|_{H^{\sigma_1,0}}
\exp\bigl(CA_\d(\uh_0)\bigr)
\end{equation} for $A_\d (\uh_0)$ being determined by \eqref{S1eq3}.
}\end{prop}
The definitions of the functional spaces
will be presented in Section \ref{Sect3}.

It is easy to check that  the system satisfied
by the difference $u-\baru$, where $\baru=(\baru^\h,0)$
\footnote{This is just for convenience of notations, and
one should keep in mind that $\baru^3=0.$} contains  quadric term
$\baruh\cdot\nablah u^3,$ which is not small.
The way to round this difficulty is to introduce a correction velocity, $\wt u$,
to be determined by
 \begin{equation}\label{eqtwtu}
\left\{\begin{array}{l}
\displaystyle \pa_t \wt u +\baruh\cdot\nablah\wt u
-\Delta \wt u=-\nabla \wt p,\\
\displaystyle \dive \wt u = 0, \\
\displaystyle \wt u^\h|_{t=0}=\wt{u}_0^\h=-\nablah\Deltah^{-1}(\pa_3 u^3_0),\quad \wt u^3|_{t=0}
=\wt u^3_0=u^3_0.
\end{array}\right.
\end{equation}
We emphasize that it is crucial to prove
that under the smallness condition \eqref{S1eq1},
$\wt{u}^\h$ is indeed small in
some critical spaces. The main difficulty lies in the estimate of the pressure term $\na_\h \wt{p}$.
As a matter of fact, by taking space divergence to \eqref{eqtwtu} and
using the condition that $\diveh\baruh=\dive\wt{u}=0,$ we obtain
\beno
-\D\wt{p}=\diveh\bigl(\wt{u}^\h\cdot\na_h\bar{u}^\h+\wt{u}^3\pa_3\bar{u}^\h\bigr).
\eeno
We decompose the pressure $\wt{p}$ into
$\wt{p}_1+\wt{p}_2$ with
\beq\label{S2eq10}
\wt{p}_1\eqdefa (-\D)^{-1}\diveh\bigl(\wt{u}^\h\cdot\na_\h\bar{u}^\h\bigr) \andf
\wt{p}_2\eqdefa (-\D)^{-1}\diveh\bigl(\wt{u}^3\pa_3\bar{u}^\h\bigr).
\eeq
In particular, with $\pa_3\baruh$ being sufficiently small,
we can prove that $\na_\h \wt{p}$ is indeed small in the case when
$\wt{u}^\h$ is small. Therefore,
we can propagate the smallness condition for $\wt{u}^\h(t)$ for $t>0.$

Concerning the linear system \eqref{eqtwtu}, we have the following {\it a priori} estimates:

\begin{prop}\label{propwtu}
{\sl Let $\baruh$ be the global solution of \eqref{eqtbaru} determined by Proposition  \ref{propbaru}. Let $\wt{u}$
be a smooth enough solution of \eqref{eqtwtu}.  Then for any $\s_1\in ]-1,1[,$  $\s_2\in ]0,1[,$  and for any $t>0,$ we have
\begin{equation}\begin{split}\label{S2eq1}
\|\wtu\|_{{L}^\infty_t(B^{\s_1,\s_2}_{2,1})}
+\|\nabla\wtu\|_{{L}^2_t(B^{\s_1,\s_2}_{2,1})}
\leq & C\|u_0\|_{\cB^{\s_1,0}}^{1-\s_2}\|\pa_3u_0\|_{\cB^{\s_1,0}}^{\s_2}
\exp\bigl(CA_\d(\uh_0)\bigr),
\end{split}\end{equation}
\begin{equation}\begin{split}\label{S2eq2}
\|\pa_3\wtu\|_{{L}^\infty_t(H^{\s_1,0})}
+\|\na\pa_3 \wtu\|_{{L}^2_t(H^{\s_1,0})}
\leq C\|\pa_3u_0\|_{H^{\s_1,0}}\exp\bigl(CA_\d(\uh_0)\bigr),
\end{split}\end{equation}
and
\beq \label{S2eq3}
\begin{split}
\|\wt{u}^\h\|_{L^\infty_t(H^{\s_2,0})}^2+\|\na\wt{u}^\h\|_{L^\infty_t(H^{\s_2,0})}^2
\leq C\|\pa_3&u_0\|_{H^{\s_2-1,0}}^2\\
&\times\exp\Bigl(C\|u_0\|_{\cB^{0,0}}\|\pa_3u_0\|_{\cB^{0,0}}
\exp\bigl(CA_\d(\uh_0)\Bigr).
\end{split}
\eeq
}\end{prop}

The proofs of the above two propositions will be presented in Section \ref{sec4}.

\smallbreak

With $\baru$ and $\wtu$ being determined respectively by the systems \eqref{eqtbaru} and \eqref{eqtwtu}, we can split the solution $(u,p)$ of $(NS)$ as
\begin{equation}\label{decomsol}
u=\bar u+\wt u+v,\quad p=\bar p+\wt p+q,
\end{equation}
It is easy to verify that the remainder term $(v,q)$
satisfies
\begin{equation}\label{eqtv}
\left\{\begin{array}{l}
\displaystyle \pa_t v^\h +(v+\wt u)\cdot\nabla u^\h
+\baruh\cdot\nablah v^\h
-\Delta v^\h=-\nablah q,\\
\displaystyle \pa_t v^3 +(v+\wt u)\cdot\nabla u^3
+\baruh\cdot\nablah v^3
-\Delta v^3=-\pa_3 (p-\wt p),\\
\displaystyle \dive v=0, \\
\displaystyle v|_{t=0}=0.
\end{array}\right.
\end{equation}

Let $\omega\eqdef\curlh\vh.$ Motivated by \cite{CZ5}, we can equivalently reformulate \eqref{eqtv} as
\begin{equation}\label{eqtv3om}
\left\{\begin{split}
&\pa_t\omega+(v+\wt u)\cdot\nabla\Omega^\h
+\baruh\cdot\nablah\omega-\Delta \omega\\
&\qquad\qquad\qquad=\Omega^\h\pa_3u^3
+\pa_2u^3\pa_3u^1-\pa_1u^3\pa_3u^2
+\pa_2\baruh\cdot\nablah\wt u^1-\pa_1\baruh\cdot\nablah\wt u^2, \\
&\pa_t v^3 +(v+\wt u)\cdot\nabla u^3
+\baruh\cdot\nablah v^3-\Delta v^3\\
&\qquad\qquad\qquad
=-\partial_3\D^{-1} \Bigl(\sum_{\ell,m=1}^3 \partial_\ell
u^m\partial_m u^\ell
-\sum_{\ell,m=1}^3\pa_\ell\baru^m\pa_m\wt u^\ell\Bigr),\\
& \omega|_{t=0}=0,~v^3|_{t=0}=0,
\end{split}\right.
\end{equation}
where $\Omega^\h\eqdef\curlh u^\h=\omega+\bar{\om}+\wtom$ and $\barom\eqdef\curlh\baruh,~\wtom\eqdef\curlh
\wt u^\h.$

Notice that the right-hand side of the $\om$ equation in \eqref{eqtv3om} contains terms either with $\pa_3u$ or $\wt{u}^\h,
$ which are small according to Propositions   \ref{propbaru} and \ref{propwtu}. Similar observation holds for the source terms in the $v^3$ equation of \eqref{eqtv3om}. Therefore, we have
reason to expect that both $\om$ and $v^3$ can exist and keep being small
in some critical spaces for all time.
To rigorously justify this expectation, we deduce from the discussions at the beginning of this section and
Propositions \ref{propbaru} and \ref{propwtu} that
 \eqref{eqtv} also has a unique maximal solution
$v\in \cE_{T^\ast}$, which is determined by \eqref{FTS}.

Let us denote
\begin{equation}\begin{split}\label{defMN}
&M(t)\eqdef \|\nabla v^3(t)\|_{H^{-\f12,0}}^2
+\|\omega(t)\|_{H^{-\f12,0}}^2
,\quad
N(t)\eqdef \|\nabla^2 v^3(t)\|_{H^{-\f12,0}}^2
+\|\nabla\omega(t)\|_{H^{-\f12,0}}^2.
\end{split}\end{equation}
We shall prove the following propositions in Section \ref{sec5}.

\begin{prop}\label{aprioriv}
{\sl  For any $t<T^\ast$, the maximal solution $v\in \cE_{T^\ast}$
of \eqref{eqtv} satisfies
\begin{equation}\begin{split}\label{ineqv3}
& \f{d}{dt}\|\nabla v^3\|_{\hm}^2+2\|\nabla^2 v^3\|_{\hm}^2
\leq\bigl(\f14+CM^{\f12}\bigr)N
+CM\bigl(\|\nabla\wt u\|_\bh^2+\|\nabla\baru^\h\|_\bh^2\bigr)\\
&+C\|\wt u^\h\|_{H^{\f12,0}}^2\|\nabla\wt u^3\|_\bh^2
+C\bigl(\|\nabla\pa_3\baruh\|_{H^{-\f12,0}}^2
+\|\nabla\pa_3\wt u\|_{H^{-\f12,0}}^2\bigr)
\bigl(\|\baruh\|_\bh^2+\|\wt u\|_\bh^2\bigr),
\end{split}\end{equation}
and
\begin{equation}\begin{split}\label{ineqom}
\f{d}{dt}\|\om& \|_{\hm}^2+2\|\nabla\om\|_{\hm}^2
\leq\bigl(\f14+CM^{\f12}+M\bigr)N
+CM\bigl(\|\nabla\wt u\|_\bh^2+\|\nabla\baru^\h\|_\bh^2\bigr)\\
&+C\|\nabla\wt u^\h\|_\hh^2
+C\bigl(\|\nabla\pa_3\baruh\|_{H^{-\f12,0}}^2
+\|\nabla\pa_3\wt u\|_{H^{-\f12,0}}^2\bigr)
\bigl(1+\|\baruh\|_\bh^2+\|\wt u\|_\bh^2\bigr).
\end{split}\end{equation}
}\end{prop}

In Section \ref{sec6}, we shall conclude the proof of Theorem \ref{thmmain}.
The strategy of the proof is as follows

\begin{prop}\label{S6prop1}
{\sl Under the assumptions of Theorem \ref{thmmain}, there exists some
positive constant $\eta$ such that
\beq\label{S6eq4}
\sup_{t\in [0,T^\ast[}\Bigl(M(t)+\int_0^tN(t')\,dt'\Bigr)
\leq \eta.
\eeq
}\end{prop}
With this estimate at hand, we then appeal to the following regularity criteria for the local Fujita-Kato solution of $(NS)$:
\begin{thm}[Theorem 1.5 of \cite{CZ5}]
\label{blowupBesovendpoint}
{\sl Let $u$ be a solution of~$(NS)$ in the
space~$\cE_{T^\ast}$. If the maximal existence time $T^\ast$ is finite,
then for any~$(p_{i,j})$ in~$]1,\infty[^9$, one has
\begin{equation}\label{blowupCZ5}
\sum_{1\leq i,j\leq3} \int_0^{T^\ast} \|\partial_{i}
u^{j}(t)\|^{p_{i,j}}_{B_{\infty,\infty}
^{-2+\f2{p_{i,j}}}} \,dt=\infty.
\end{equation}
}\end{thm}

\setcounter{equation}{0}
\section{Anisotropic Littlewood-Paley Theory}\label{Sect3}

In this section, we shall collect some basic facts on anisotropic Littlewood-Paley theory from \cite{bcdbookk}.
Let us first recall the following anisotropic dyadic operators:
\begin{equation}\begin{split}\label{defparaproduct}
&\Delta_k^{\rm h}a=\cF^{-1}\bigl(\varphi(2^{-k}|\xi_{\rm h}|)\widehat{a}(\xi)\bigr),
 \quad \Delta_\ell^{\rm v}a =\cF^{-1}\bigl(\varphi(2^{-\ell}|\xi_3|)\widehat{a}(\xi)\bigr),\\
&S^{\rm h}_ka=\cF^{-1}\bigl(\chi(2^{-k}|\xi_{\rm h}|)\widehat{a}(\xi)\bigr),
\quad\ S^{\rm v}_\ell a =\cF^{-1}\bigl(\chi(2^{-\ell}|\xi_3|)\widehat{a}(\xi)\bigr),
\end{split}\end{equation}
where $\xi=(\xi_\h,\xi_3)$ and $\xi_{\rm h}=(\xi_1,\xi_2),$ $\cF a$ and
$\widehat{a}$ denote the Fourier transform of $a$,
while $\cF^{-1} a$ denotes its inverse Fourier transform,
$\chi(\tau)$ and $\varphi(\tau)$ are smooth functions such that
\begin{align*}
&\Supp \varphi \subset \Bigl\{\tau \in \R\,: \, \frac34 \leq
|\tau| \leq \frac83 \Bigr\}\quad\mbox{and}\quad \forall
 \tau>0\,,\ \sum_{j\in\Z}\varphi(2^{-j}\tau)=1,\\
& \Supp \chi \subset \Bigl\{\tau \in \R\,: \, |\tau| \leq
\frac43 \Bigr\}\quad\mbox{and}\quad \forall
 \tau\in\R\,,\ \chi(\tau)+ \sum_{j\geq 0}\varphi(2^{-j}\tau)=1.
\end{align*}

Let us recall the  anisotropic Bernstein inequalities from \cite{CZ1, Pa02}.

\begin{lem}\label{lemBern}
{\sl Let $\cB_{\h}$ (resp.~$\cB_{\rm v}$) a ball
of~$\R^2_{\h}$ (resp.~$\R_{\rm v}$), and~$\cC_{\h}$ (resp.~$\cC_{\rm v}$) a
ring of~$\R^2_{\h}$ (resp.~$\R_{\rm
v}$); let~$1\leq p_2\leq p_1\leq
\infty$ and ~$1\leq q_2\leq q_1\leq \infty.$ Then there holds:
\beno
\begin{split}
\mbox{if}\ \ \Supp \wh a\subset 2^k\cB_{\h}&\Rightarrow
\|\partial_{x_{\rm h}}^\alpha a\|_{L^{p_1}_{\rm h}(L^{q_1}_{\rm v})}
\lesssim 2^{k\left(|\alpha|+2\left(1/{p_2}-1/{p_1}\right)\right)}
\|a\|_{L^{p_2}_{\rm h}(L^{q_1}_{\rm v})};\\
\mbox{if}\ \ \Supp\wh a\subset 2^\ell\cB_{\rm v}&\Rightarrow
\|\partial_{x_3}^\beta a\|_{L^{p_1}_{\rm h}(L^{q_1}_{\rm v})}
\lesssim 2^{\ell\left(\beta+(1/{q_2}-1/{q_1})\right)} \|
a\|_{L^{p_1}_{\rm h}(L^{q_2}_{\rm v})};\\
\mbox{if}\ \ \Supp\wh a\subset 2^k\cC_{\h}&\Rightarrow
\|a\|_{L^{p_1}_{\rm h}(L^{q_1}_{\rm v})} \lesssim
2^{-kN}\sup_{|\alpha|=N}
\|\partial_{x_{\rm h}}^\alpha a\|_{L^{p_1}_{\rm
h}(L^{q_1}_{\rm v})};\\
\mbox{if}\ \ \Supp\wh a\subset2^\ell\cC_{\rm v}&\Rightarrow
\|a\|_{L^{p_1}_{\rm h}(L^{q_1}_{\rm v})} \lesssim 2^{-\ell N}
\|\partial_{x_3}^N a\|_{L^{p_1}_{\rm h}(L^{q_1}_{\rm v})}.\end{split}\eeno
}
\end{lem}

\begin{defi}\label{anibesov}
{\sl Let us define the anisotropic Besov space $B^{s_1,s_2}_{p,q}$  (with  usual adaptation when $q$  equal $\infty$) as the space of
homogenous  tempered distributions $u$ so that
$$
\|u\|_{B^{s_1,s_2}_{p,q}}\eqdef \Bigl(\sum_{k\in\Z}\sum_{\ell\in\Z}
2^{qks_1}2^{q\ell s_2}\|\D_k^{\rm h}
\D_\ell^{\rm v}u\|_{L^p}^{q}\Bigr)^{1/{q}}<\infty.
$$
}\end{defi}

We remark that
$B^{s_1,s_2}_{2,2}$ coincides with the classical anisotropic Sobolev space $H^{s_1,s_2}$ given by Definition \ref{defanisob}.

\begin{defi}\label{def2}
{\sl  Let $s\in\R,$ $p\in [1,\infty]$ and
$a\in{\mathcal S}_h'(\R^3),$ we define
$$\|a\|_{\cB^{s,0}}\eqdef\sum_{k\in\Z}2^{k s}\|\D_k^\h a\|_{L^2},$$
and the corresponding Chemin-Lerner type norm  (see \cite{CL})
$$\|a\|_{\wt{L}^p_T(\cB^{s,0})}\eqdefa
 \sum_{k\in\Z}2^{k s}\|\D_k^\h a\|_{L^p_T(L^2)}.$$
}\end{defi}

\begin{rmk} For any $a\in H^{s,0},$ we deduce from Fourier-Plancherel inequality that
\beno
\begin{split}
\|\D_k^\h a\|_{L^2}=&\Bigl(\sum_{\ell\in\Z}\|\D_k^\h\D_\ell^\v a\|_{L^2}^2\Bigr)^{\f12}\\
\lesssim &\Bigl(\sum_{\ell\in\Z}c_{k,\ell}^22^{-2ks}\Bigr)^{\f12}\|a\|_{H^{s,0}}\lesssim c_k2^{-ks}\|a\|_{H^{s,0}},
\end{split}
\eeno where $\left(c_{k,\ell}\right)_{(k,\ell)\in\Z^2}$ (resp.  $\left(c_{k}\right)_{k\in\Z}$ ) is a generic element of
$\ell^2(\Z^2)$ (resp. $\ell^2(\Z^2)$) so that $\sum_{(k,\ell)\in\Z^2}c_{k,\ell}^2=1$ (resp. $\sum_{k\in\Z}c_{k}^2=1$).

Then for any $s\in ]s_1,s_2[$ and for any integer $N$, we find
\beno
\begin{split}
\|a\|_{\cB^{s,0}}=&\sum_{k\leq N}2^{ks}\|\D_k^\h a\|_{L^2}+\sum_{k>N}2^{ks}\|\D_k^\h a\|_{L^2}\\
\leq &\sum_{k\leq N}c_k 2^{k(s-s_1)}\|a\|_{H^{s_1,0}}+\sum_{k>N}2^{-k(s_2-s)}\|a\|_{H^{s_2,0}}\\
\lesssim & 2^{N(s-s_1)}\|a\|_{H^{s_1,0}}+2^{-N(s_2-s)}\|a\|_{H^{s_2,0}}.
\end{split}
\eeno
Taking the integer $N$ so that
$
2^{N(s_2-s_1)}\sim \f{\|a\|_{H^{s_2,0}}}{\|a\|_{H^{s_1,0}}}
$
leads to
\beq \label{S3eq6}
\|a\|_{\cB^{s,0}}\leq C\|a\|_{H^{s_1,0}}^{\f{s_2-s}{s_2-s_1}}\|a\|_{H^{s_2,0}}^{\f{s-s_1}{s_2-s_1}}.
\eeq
\end{rmk}

On the other hand, to overcome the difficulty that one can not use Gronwall's inequality in the Chemin-Lerner type space, we  need the
time-weighted Chemin-Lerner norm, which was introduced by Paicu and the second  author in \cite{PZ1}:
\begin{defi}\label{defpz}
Let $f(t)\in L^1_{\rm{loc}}(\R_+)$, $f(t)\geq 0$. We denote
$$\|u\|_{\widetilde L^2_{T,f}(\cB^{s,0})}=\sum_{k\in\Z} 2^{k s}
\Bigl(\int_0^Tf(t)\|\Delta_k^\h u(t)\|_{L^2}^2\,dt\Bigr)^{\frac 12}.$$
\end{defi}

We also need  Bony's decomposition from \cite{Bo81} for the horizonal variables to study the law of product in Besov spaces:
\begin{equation}\label{bony}\begin{split}
& ab=T^\h_a b+{T}^\h_b a+ R^\h(a,b)\quad\mbox{with}\\
T^\h_a b=\sum_{j\in\Z}S^\h_{j-1}a\Delta^\h_jb,\quad&
R^\h(a,b)=\sum_{j\in\Z}\Delta^\h_ja\widetilde{\Delta}^\h_{j}b
\quad \mbox{where}\quad \widetilde{\Delta}^\h_{j}\eqdef{\Delta}^\h_{j-1}
+{\Delta}^\h_{j}+{\Delta}^\h_{j+1}.
\end{split}\end{equation}

As an application of the above basic elements on Littlewood-Paley theory, we present
the following law of product in $\cB^{s,0}.$

\begin{lem}\label{lemproductlawBh}
{\sl For any $s_1,~s_2\leq 1$ satisfying $s_1+s_2>0$,
one has
\beq \label{S3eq3}\|ab\|_{\cB^{s_1+s_2-1,0}}\lesssim\|a\|_{B^{s_1,\f12}_{2,1}}
\|b\|_{\cB^{s_2,0}}.\eeq
}\end{lem}

\begin{proof} According to \eqref{bony}, we split the product $ab$
into three parts: $T_a^\h b,$ $T_b^\h a$ and $R^\h(a,b).$
Due to $s_1\leq 1,$ we deduce from Lemma \ref{lemBern} that
\beno
\begin{split}
\|S_k^\h a\|_{L^\infty}\lesssim& \sum_{\substack{k'\leq k-1\\ \ell\in\Z}}
2^{k'}2^{\f{\ell}2}\|\D_{k'}^\h\D_\ell^\v a\|_{L^2}\\
\lesssim & \sum_{\substack{k'\leq k-1\\ \ell\in\Z}}
d_{k',\ell}2^{k'(1-s_1)}\|a\|_{B^{s_1,\f12}_{2,1}}
\lesssim  2^{k(1-s_1)}\|a\|_{B^{s_1,\f12}_{2,1}},
\end{split}
\eeno where  $\left(d_{k,\ell}\right)_{(k,\ell)\in\Z^2}$ (resp. $\left(d_{k}\right)_{k\in\Z}$) designates a generic element on the unit
sphere of  $\ell^1(\Z^2)$  (resp. $\ell^1(\Z)$) so that  $\sum_{(k,\ell)\in\Z^2}d_{k,\ell}=1$ (resp. $\sum_{k\in\Z}d_{k}=1$).
Then by virtue of  the support properties to the Fourier transform of the terms in $T^\h_ab,$ we infer
\beno\begin{split}
\|\dhk T^\h_a b\|_{L^2}
&\lesssim\sum_{|k'-k|\leq4}
\|S^\h_{k'-1}a\|_{L^\infty}\|\dhkp b\|_{L^2}\\
&\lesssim\sum_{|k'-k|\leq 4}d_{k'}2^{k'(1-s_1-s_2)}\|a\|_{B^{s_1,\f12}_{2,1}}\|b\|_{\cB^{s_2,0}}\\
&\lesssim d_{k}2^{k(1-s_1-s_2)}\|a\|_{B^{s_1,\f12}_{2,1}}\|b\|_{\cB^{s_2,0}}.
\end{split}
\eeno
Whereas it follows from Lemma \ref{lemBern} that
\beq \label{S3eq4}
\|\D_k^\h a\|_{L^\infty_\v(L^2_\h)}\lesssim\sum_{\ell\in \Z}2^{\f{\ell}2}\|\D_k^\h\D_\ell^\v a\|_{L^2}
\lesssim\sum_{\ell\in \Z}d_{k,\ell}2^{ks_1}\|a\|_{B^{s_1,\f12}_{2,1}}\lesssim d_k2^{ks_1}\|a\|_{B^{s_1,\f12}_{2,1}},
\eeq
and due to $s_2\leq 1,$ we have $\|S_k^\h b\|_{L^2_\v(L^\infty_\h)}\lesssim 2^{k(1-s_2)}\|b\|_{\cB^{s_2,0}}.$ As a result, it comes out
\begin{equation*}\begin{split}\label{3.8}
\|\dhk T^\h_b a\|_{L^2}
&\lesssim\sum_{|k'-k|\leq 4}\|S^\h_{k'-1}b\|_{L^2_\v(L_\h^\infty)}
\|\dhkp a\|_{L^\infty_\v(L_\h^2)}\\
&\lesssim\sum_{|k'-k|\leq 4}d_{k'}2^{k'(1-s_1-s_2)}\|a\|_{B^{s_1,\f12}_{2,1}}\|b\|_{\cB^{s_2,0}}\\
&\lesssim d_{k}2^{k(1-s_1-s_2)}\|a\|_{B^{s_1,\f12}_{2,1}}\|b\|_{\cB^{s_2,0}}.
\end{split}\end{equation*}
For the remainder term, in view of \eqref{S3eq4} and the assumption that
$s_1+s_2>0$, we find
\begin{align*}
\|\dhk R^\h(a,b)\|_{L^2}
&\lesssim 2^k\sum_{k'\geq k-3}
\|\dhkp a\|_{L_\v^\infty(L^2_h)}
\|\wt\D^\h_{k'} b\|_{L^2}\\
&\lesssim 2^k\sum_{k'\geq k-3}d_{k'}2^{-k'(s_1+s_2)} \|a\|_{B^{s_1,\f12}_{2,1}}\|b\|_{\cB^{s_2,0}}
\\
&\lesssim d_{k}2^{k(1-s_1-s)}\|a\|_{B^{s_1,\f12}_{2,1}}\|b\|_{\cB^{s_2,0}}.
\end{align*}

By summing up the above estimates, we conclude the proof of \eqref{S3eq3}.
\end{proof}

The following law of product will be frequently used
in the rest of this  paper.

\begin{lem}[A special case of Lemma 4.5 in \cite{CZ5}]\label{lemproductlaw}
{\sl For any $s_1,~s_2<1,~s_3\leq 1$ with $s_1+s_2>0,~s_1+s_3>0$,
 and for any $r_1,~r_2<\f12,~r_3\leq\f12$ with $r_1+r_2>0,~r_1+r_3>0$,
we have
\begin{equation}\label{lemproduct1}
\|ab\|_{H^{s_1+s_2-1,\tau_1+\tau_2-\f12}}\lesssim\|a\|_{H^{s_1,\tau_1}}
\|b\|_{H^{s_2,\tau_2}},
\end{equation}
\begin{equation}\label{lemproduct2}
\|ab\|_{H^{s_1+s_3-1,\tau_1+\tau_3-\f12}}\lesssim\|a\|_{H^{s_1,\tau_1}}
\|b\|_{B^{s_3,\tau_3}_{2,1}}.
\end{equation}
}\end{lem}

We remark that \eqref{lemproduct2} can be seen as a borderline
case of \eqref{lemproduct1}.  However, this small generalization
can be crucial in this context. Indeed, considering the initial data given by
\eqref{slowinitialdata}, we observe that
$$\|v^\h_0(x_\h,\ve x_3)\|_B=\ve^{r-\f12}\|v^\h_0(x)\|_B,
\ \mbox{where}\  B=H^{s,r}\ \mbox{or}\ B^{s,r}_{2,1}.$$
So that to keep the uniform boundedness of $\|v^\h_0(x_\h,\ve x_3)\|_B,$ we must require $r\geq \frac12.$
Thanks to the product law in the homogeneous Besov spaces for the vertical variable, we can choose $r$ to be nothing but $\f12.$

\medskip

\setcounter{equation}{0}
\section{{\it A priori} estimates of  the systems \eqref{eqtbaru} and \eqref{eqtwtu}}\label{sec4}

The goal of this section is to present the proof of Propositions \ref{propbaru} and \ref{propwtu}.

\subsection{Two useful lemmas}

The key ingredients used in the proof of Propositions \ref{propbaru} and \ref{propwtu} are the subsequent two lemmas:

\begin{lem}\label{S4lem1}
{\sl Let $f(t)\eqdefa \|\na_\h\baruh(t)\|_{L^\infty_{\v}(L^2_\h)}^2$ and $s\in ]-1,2[,$ we have
\beq
\label{S4eq6}
\int_0^t\bigl|\bigl(\D_k^\h (\baruh\cdot\nablah b) | \D_k^\h b\bigr)_{L^2}\bigr|\,dt'\lesssim d_k^22^{-2ks}\|b\|_{\wt{L}^2_{t,f}(\cB^{s,0})}\|\na_\h b\|_{\wt{L}^2_t(\cB^{s,0})}.
\eeq}
\end{lem}

\begin{proof}
Applying Bony's decomposition on the horizontal variables, \eqref{bony}, to
$\baruh\cdot\na_\h b$  gives
\beq \label{S4eq7}
\baruh\cdot\na_\h b=T^\h_{\baruh}\na_\h b
+T^\h_{\na_\h b} \baruh+R^\h(\baruh,\na_\h b).
\eeq
Considering the support properties to the Fourier transform of the terms in $T^\h_{\baruh}\na_\h b,$
 we get, by applying a commutator's argument, that
\beno
\begin{split}
\bigl(\D_k^\h (T^\h_{\baruh}\na_\h b) \big| \D_k^\h b\bigr)_{L^2}=&\sum_{|k'-k|\leq 4}
\Bigl(\bigl([\D_k^\h; S_{k'-1}^\h\baruh]\D_{k'}^\h\na_\h b
\big|  \D_k^\h b\bigr)_{L^2}\\
&\qquad\qquad+\bigl(( S_{k'-1}^\h\baruh-S_{k-1}^\h\baruh)\D_k^\h\D_{k'}^\h\na_\h  b
\big|  \D_k^\h  b\bigr)_{L^2}\Bigr)\\
&+\bigl(S_{k-1}^\h\baruh\cdot\na_\h\D_k^\h  b
\big|  \D_k^\h b\bigr)_{L^2}\eqdefa \cA_1+\cA_2+\cA_3.
\end{split}
\eeno
Applying  standard commutator's estimate (see \cite{bcdbookk})
and Lemma \ref{lemBern} yields
\beno
\begin{split}
\int_0^t|\cA_1|\,dt'\lesssim &\sum_{|k'-k|\leq 4}2^{-k'}
\int_0^t\|S_{k'-1}^\h\na_\h\baruh\|_{L^\infty}\|\D_{k'}^\h\na_\h  b\|_{L^2}\|\D_k^\h b\|_{L^2}\,dt'\\
\lesssim &\sum_{|k'-k|\leq 4}\int_0^t
\|\na_\h\baruh\|_{L^\infty_\v(L^2_\h)}\|\D_{k'}^\h\na_\h  b\|_{L^2}\|\D_k^\h b\|_{L^2}\,dt'\\
\lesssim &\sum_{|k'-k|\leq 4} \|\D_{k'}^\h\na_\h  b\|_{L^2_t(L^2)}
\Bigl(\int_0^tf(t')\|\D_k^\h b\|_{L^2}^2\,dt'\Bigr)^{\f12}\\
\lesssim &d_k^2 2^{-2ks}\|\na_\h b\|_{\wt{L}^2_t(\cB^{s,0})}
\| b\|_{\wt{L}^2_{t,f}(\cB^{s,0})}.
\end{split}
\eeno
Along the same line, by applying Lemma \ref{lemBern} once again, we obtain
\beno
\begin{split}
\int_0^t|\cA_2|\,dt'\lesssim &\sum_{|k'-k|\leq 4}2^{-k'}\int_0^t\|\D_{k'}^\h\na_\h\baruh\|_{L^\infty}\|\D_{k'}^\h\na_\h  b\|_{L^2}\|\D_k^\h b\|_{L^2}\,dt'\\
\lesssim &\sum_{|k'-k|\leq 4}\int_0^t\|\na_\h\baruh\|_{L^\infty_\v(L^2_\h)}\|\D_{k'}^\h\na_\h  b\|_{L^2}\|\D_k^\h b\|_{L^2}\,dt'\\
\lesssim &d_k^2 2^{-2ks}\|\na_\h b\|_{\wt{L}^2_t(\cB^{s,0})}
\| b\|_{\wt{L}^2_{t,f}(\cB^{s,0})}.
\end{split}
\eeno
Whereas due to $\dive_\h \baruh=0,$ we have $\cA_3=0.$ This in turn shows that
\beq\label{4.5}
\int_0^t\bigl|\bigl(\D_k^\h(T^\h_{\baruh}\na_\h  b)
\big| \D_k^\h b\bigr)_{L^2}\bigr|\,dt'\lesssim d_k^2 2^{-2ks}\|\na_\h b\|_{\wt{L}^2_t(\cB^{s,0})}
\| b\|_{\wt{L}^2_{t,f}(\cB^{s,0})}.
\eeq

On the other hand, we get, by applying  Lemma \ref{lemBern}, that
\beno
\begin{split}
\int_0^t\bigl|\bigl(\D_k^\h(R^\h({\baruh},\na_\h  b)
\big| \D_k^\h b\bigr)_{L^2}\bigr|\,dt'
&\lesssim
2^k\sum_{k'\geq k-3}\int_0^t\|\wt{\D}_{k'}^\h\baruh\|_{L^\infty_\v(L^2_\h)}
\|{\D}^\h_{k'}\na_\h b\|_{L^2}\|\D_k^\h b\|_{L^2}\,dt'\\
&\lesssim2^k\sum_{k'\geq k-3}2^{-k'}\int_0^t
\|\na_\h \baruh\|_{L^\infty_\v(L^2_\h)}
\|{\D}^\h_{k'}\na_\h b\|_{L^2}\|\D_k^\h b\|_{L^2}\,dt'\\
&\lesssim2^k\sum_{k'\geq k-3}2^{-k'}\|\D_{k'}^\h\na_\h  b\|_{L^2_t(L^2)}\Bigl(\int_0^tf(t')
\|\D_k^\h b\|_{L^2}^2\,dt'\Bigr)^{\f12}\\
&\lesssim d_k 2^{-2ks}\sum_{k'\geq k-3} d_{k'} 2^{-(k'-k)(1+s)}
\|\na_\h b\|_{\wt{L}^2_t(\cB^{s,0})}
\| b\|_{\wt{L}^2_{t,f}(\cB^{s,0})}.
\end{split}
\eeno
Then using the fact that $s>-1,$ we achieve
\beq \label{4.6}
\begin{split}
\int_0^t\bigl|\bigl(\D_k^\h(R^\h({\baruh},\na_\h  b)
 \big| \D_k^\h b\bigr)_{L^2}\bigr|\,dt'
\lesssim d_k^22^{-2ks}\|\na_\h b\|_{\wt{L}^2_t(\cB^{s,0})}
\| b\|_{\wt{L}^2_{t,f}(\cB^{s,0})}.
\end{split}
\eeq

It remains to handle the estimate of $\int_0^t\bigl|\bigl(\D_k^\h(T^\h_{\na_\h b} \baruh) \big|
\D_k^\h b\bigr)_{L^2}\bigr|\,dt'.$ Indeed applying  Lemma \ref{lemBern} yields
\beno
\begin{split}
\int_0^t&\bigl|\bigl(\D_k^\h(T^\h_{\na_\h b} \baruh) \big|
\D_k^\h b\bigr)_{L^2}\bigr|\,dt'\\
\lesssim &2^{-k}\sum_{|k'-k|\leq 4}2^{-k'}\int_0^t\|S_{k'-1}^\h\na_\h b\|_{L^2_\v(L^\infty_\h)}\|\D_{k'}^\h \na_\h\baruh
\|_{L^\infty_\v(L^2_\h)}\|\D_{k}^\h \na_\h b\|_{L^2}\,dt'\\
\lesssim &2^{-2k} \sum_{|k'-k|\leq 4}
\Bigl(\int_0^t\|\na_\h\baruh\|_{L^\infty_\v(L^2_\h)}^2
\|S_{k'-1}^\h\na_\h b\|_{L^2_\v(L^\infty_\h)}^2\,dt'\Bigr)^{\f12}
\|\D_{k}^\h \na_\h b\|_{L^2_t(L^2)}.
\end{split}
\eeno
Yet due to $s<2,$ we have
\beq\label{S4eq3}
\begin{split}
\Bigl(\int_0^tf(t')
\|S_{k'-1}^\h\na_\h b\|_{L^2_\v(L^\infty_\h)}^2\,dt'\Bigr)^{\f12}\lesssim &\sum_{k''\leq k'-2}2^{2k''}\Bigl(\int_0^tf(t')
\|\D_{k''}^\h b\|_{L^2}^2\,dt'\Bigr)^{\f12}\\
\lesssim &\sum_{k''\leq k'-2}d_{k''}2^{k''(2-s)}\|b\|_{\wt{L}^2_{t,f}(\cB^{s,0})}\\
\lesssim &d_{k'}2^{k'(2-s)}\|b\|_{\wt{L}^2_{t,f}(\cB^{s,0})}.
\end{split}
\eeq
As a result, it comes out
\begin{equation}\label{4.7}
\int_0^t\bigl|\bigl(\D_k^\h(T^\h_{\na_\h b} \baruh) \big|
 \D_k^\h b\bigr)_{L^2}\bigr|\,dt'
 \lesssim d_k^2 2^{-2ks}\|\na_\h b\|_{\wt{L}^2_t(\cB^{s,0})}
\| b\|_{\wt{L}^2_{t,f}(\cB^{s,0})}.
\end{equation}
This together with \eqref{4.5} and \eqref{4.6} ensures  \eqref{S4eq6}  for any $s\in ]-1,2[.$
\end{proof}

\begin{rmk}\label{S4rmk1}
We notice that the condition $s<2$ is only used in the proof of \eqref{4.7}.
In the case when $b=\baruh,$ we can handle the
estimate \eqref{4.7} alternatively as follows
\beno
\begin{split}
\int_0^t\bigl|&\bigl(\D_k^\h(T^\h_{\na_\h \baruh} \baruh) \big|
\D_k^\h \baruh\bigr)_{L^2}\bigr|\,dt'\\
&\lesssim 2^{-k}\sum_{|k'-k|\leq 4}2^{k'}\int_0^t\|S_{k'-1}^\h\na_\h \baruh\|_{L^\infty_\v(L^2_\h)}
\|\D_{k'}^\h \baruh\|_{L^2}\|\D_{k}^\h \na_\h \baruh\|_{L^2}\,dt'\\
&\lesssim\sum_{|k'-k|\leq 4}\Bigl(\int_0^t\|\na_\h\baruh
\|_{L^\infty_\v(L^2_\h)}^2\|\D_{k'}\baruh\|_{L^2}^2\,dt'\Bigr)^{\f12}\|\D_{k}^\h \na_\h \baruh\|_{L^2_t(L^2)}\\
&\lesssim d_k^2 2^{-2ks}\|\baruh\|_{\wt{L}^2_{t,f}(\cB^{s,0})}
\|\na_\h \baruh\|_{\wt{L}^2_t(\cB^{s,0})},
\end{split}
\eeno
which does not  require $s<2.$ Therefore, \eqref{S4eq6} holds for any $s>-1$ in the case when
 $b=\baruh$~(actually
when $b=\baruh\cdot g(t)$ for any function $g(t)$).
\end{rmk}

\begin{lem}\label{S4lem2}
{\sl Let  $f(t)\eqdefa \|\na_\h\baruh(t)\|_{L^\infty_{\v}(L^2_\h)}^2$ and $s\in ]-1,1[,$ one has
\beq \label{S4eq8}
\int_0^t\bigl|\bigl(\D_k^\h( b\cdot\na_\h\baruh)
\big| \D_k^\h b\bigr)_{L^2}\bigr|\,dt'
\lesssim d_k^2 2^{-2ks}\|\na_\h b\|_{\wt{L}^2_t(\cB^{s,0})}
\|b\|_{\wt{L}^2_{t,f}(\cB^{s,0})}.
\eeq}
\end{lem}

\begin{proof}
By  applying
Bony's decomposition  for the horizontal variables, \eqref{bony},  to
$b\cdot\na_\h\baruh,$ we write
$$ b\cdot\na_\h\baruh=T^\h_{b}\nablah \baruh
+T^\h_{\nablah\baruh}b
+R^\h(b,\nablah\baruh).$$
It follows from Lemma \ref{lemBern} that
\begin{equation*}\begin{split}
& \int_0^t\bigl|\bigl(\D_k^\h (T^\h_{b}\na_\h\baruh)
\big| \D_k^\h b\bigr)_{L^2}\bigr|\,dt'\\
&\lesssim 2^{-k} \sum_{|k'-k|\leq 4}
\int_0^t\|S_{k'-1}^\h b\|_{L^2_\v(L^\infty_\h)}
\|\D_{k'}^\h\na_\h\baruh\|_{L^\infty_\v(L^2_\h)}
\|\D_k^\h\na_\h b\|_{L^2}\,dt'\\
&\lesssim 2^{-k} \sum_{|k'-k|\leq 4} \Bigl(\int_0^t\|\na_\h\baruh\|_{L^\infty_\v(L^2_\h)}^2
\|S_{k'-1}^\h b\|_{L^2_\v(L^\infty_\h)}^2\,dt'\Bigr)^{\f12}
\|\D_k^\h\na_\h b\|_{L^2_t(L^2)},
\end{split}\end{equation*}
Yet it follows from the derivation of \eqref{S4eq3} that
\beno
\Bigl(\int_0^t\|\na_\h\baruh\|_{L^\infty_\v(L^2_\h)}^2
\|S_{k'-1}^\h b\|_{L^2_\v(L^\infty_\h)}^2\,dt'\Bigr)^{\f12}
\lesssim d_{k'}  2^{k'(1-s)}\|b\|_{\wt{L}^2_{t,f}(\cB^{s,0})}\qquad\forall\ s<1,
\eeno
which implies for any $s<1$ that
\begin{equation}\label{4.17}
\int_0^t\bigl|\bigl(\D_k^\h (T^\h_{b}\na_\h\baruh)
\big| \D_k^\h b\bigr)_{L^2}\bigr|\,dt'
\lesssim d_k^2 2^{-2ks}\|b\|_{\wt{L}^2_{t,f}(\cB^{s,0})}\|\na_\h b\|_{\wt{L}^2_t(\cB^{s,0})}.
\end{equation}

Applying Lemma \ref{lemBern} twice yields
\begin{equation}\label{4.18}\begin{split}
\int_0^t&\bigl|\bigl(\D_k^\h (T^\h_{\nablah\baruh}b)
| \D_k^\h b\bigr)_{L^2}\bigr|\,d t'\\
&\lesssim\sum_{|k'-k|\leq 4}\int_0^t\|S_{k'-1}^\h\na_\h \baruh\|_{L^\infty}
\|\D_{k'}^\h b\|_{L^2}\|\D_k^\h b\|_{L^2}\,dt'\\
&\lesssim\sum_{|k'-k|\leq 4}\|\D_{k'}^\h\na_\h b\|_{L^2_t(L^2)}\Bigl(\int_0^t\|\na_\h\baruh\|_{L^\infty_\v(L^2_\h)}^2
\|\D_k^\h b\|_{L^2}\,dt'\Bigr)^{\f12}\\
&\lesssim d_k^2 2^{-2ks}\|\na_\h b\|_{\wt{L}^2_t(\cB^{s,0})}
\| b\|_{\wt{L}^2_{t,f}(\cB^{s,0})}.
\end{split}\end{equation}
And the remainder term can be handled as follows
\begin{equation*}\begin{split}
\int_0^t\bigl|\bigl(\D_k^\h (R^\h(b,\nablah\baruh))
 | \D_k^\h b\bigr)_{L^2}&\bigr|\,dt'
\lesssim 2^k\sum_{k'\geq k-3}\int_0^t
\|{\D}^\h_{k'} b\|_{L^2}\|\wt{\D}_{k'}^\h\na_\h\baruh\|_{L^\infty_\v(L^2_\h)}
\|\D_k^\h b\|_{L^2}\,dt'\\
&\lesssim 2^k\sum_{k'\geq k-3}2^{-k'}\int_0^t
\|\na_\h \baruh\|_{L^\infty_\v(L^2_\h)}
\|{\D}^\h_{k'}\na_\h b\|_{L^2}
\|\D_k^\h b\|_{L^2}\,dt'\\
&\lesssim2^k\sum_{k'\geq k-3}2^{-k'}
\|\D_{k'}^\h\na_\h b\|_{L^2_t(L^2)}
\Bigl(\int_0^t f(t')\|\D_k^\h b\|_{L^2}^2\,dt'\Bigr)^{\f12}\\
&\lesssim d_k 2^{-2ks}\sum_{k'\geq k-3} d_{k'} 2^{-(k'-k)(1+s)}
\|\na_\h b\|_{\wt{L}^2_t(\cB^{s,0})}
\|b\|_{\wt{L}^2_{t,f}(\cB^{s,0})},
\end{split}\end{equation*}
which together with the fact that $s>-1$ ensures that
\begin{equation*}\label{4.19}
\int_0^t\bigl|\bigl(\D_k^\h (R^\h(\nablah\baruh, b))
 | \D_k^\h b\bigr)_{L^2}\bigr|\,dt'
\lesssim d_k^2 2^{-2ks}\|\na_\h b\|_{\wt L^2_t(\cB^{s,0})}
\|b\|_{\wt{L}^2_{t,f}(\cB^{s,0})}.
\end{equation*}
Along with  \eqref{4.17} and \eqref{4.18}, we
complete the proof of \eqref{S4eq8}.
\end{proof}

\subsection{The proof of Proposition \ref{propbaru}}

In the rest
of this section, for any $\la>0$ and function $a$, we always denote
\beq \label{4.2}
a_\la(t)\eqdefa a(t)\exp\Bigl(-\la\int_0^t f(t')\,dt'\Bigr)
\quad\mbox{with}\quad f(t)\eqdefa \|\na_\h\baruh(t)\|_{L^\infty_{\v}(L^2_\h)}^2.
\eeq

\begin{prop}\label{S4prop1}
{\sl Let $\baruh$ be a smooth enough solution of \eqref{eqtbaru} on $[0,T].$
Then for any $t\in [0,T]$ and $s\in ]-1,\infty[,$ we have
\beq \label{S4eq2}\|\baruh\|_{\wt{L}^\infty_t(\cB^{s,0})}+\|\na \baruh\|_{\wt{L}^2_t(\cB^{s,0})}
\leq 2\|\uh_0\|_{\cB^{s,0}}
\exp\bigl(C A_\d(u_0^\h)\bigr),
\eeq  where $A_\d (\uh_0)$ is determined by \eqref{S1eq3}.}
\end{prop}

\begin{proof}
By multiplying $\exp\Bigl(-\la\int_0^t f(t')\,dt'\Bigr)$ to
\eqref{eqtbaru}, we write
\begin{equation}\label{eqtbarulam}
\pa_t \baruh_\la+\la f(t)\baruh_\la +\baruh\cdot\nablah\baruh_\la
-\Delta \baruh_\la=-\nablah \bar p_\la.
\end{equation}
Applying $\D_k^\h$ to the above equation and taking $L^2$ inner product of the resulting equation with
$\D_k^\h\baruh_\lam,$ we obtain
\begin{equation}\label{4.4}
\f12\f{d}{dt}\|\D_k^\h\baruh_\la(t)\|_{L^2}^2+\la f(t)\|\D_k^\h\baruh_\la(t)\|_{L^2}^2+\|\D_k^\h\na \baruh_\la\|_{L^2}^2=-\bigl(\D_k^\h(\baruh\cdot
\na_\h \baruh_\la) \big| \D_k^\h\baruh_\la\bigr)_{L^2}.
\end{equation}
By virtue of Remark \ref{S4rmk1}, for any $s>-1,$ one has
\beq \label{4.8}
\int_0^t\bigl|\bigl(\D_k^\h(\baruh\cdot\na_\h \baruh_\lam) \big|
 \D_k^\h\baruh_\la\bigr)_{L^2}\bigr|\,dt'
\lesssim d_k^2 2^{-2ks}\|\na_\h\baruh_\la\|_{\wt{L}^2_t(\cB^{s,0})}
\|\baruh_\la\|_{\wt{L}^2_{t,f}(\cB^{s,0})}.
\eeq
Then by integrating \eqref{4.4} over $[0,t]$ and inserting \eqref{4.8} into the resulting equality,
we get
\beno
\begin{split}
\|\D_k^\h\baruh_\la\|_{L^\infty_t(L^2)}^2&+\la \int_0^tf(t')\|\D_k^\h\baruh_\la(t')\|_{L^2}^2\,dt'+\|\D_k^\h\na \baruh\|_{L^2_t(L^2)}^2\\
&\lesssim  \|\D_k^\h\baruh_0\|_{L^\infty_t(L^2)}^2+d_k^2 2^{-2ks}\|\na_\h\baruh_\la\|_{\wt{L}^2_t(\cB^{s,0})}
\|\baruh_\la\|_{\wt{L}^2_{t,f}(\cB^{s,0})}.
\end{split}
\eeno
In view of Definition \ref{def2}, by multiplying the above inequality by $2^{2ks}$ and taking square root of the
resulting inequality, and then summing up the resulting inequality over $\Z,$
 we achieve
\beno
\begin{split}
\|\baruh_\la\|_{\wt{L}^\infty_t(\cB^{s,0})}
+&\sqrt{\la}\|\baruh_\la\|_{\wt{L}^2_{t,f}(\cB^{s,0})}
+\|\na \baruh_\la \|_{\wt{L}^2_t(\cB^{s,0})}\\
\leq&\|\baruh_0\|_{\cB^{s,0}}
+C\|\na_\h\baruh_\la\|_{\wt L^2_t(\cB^{s,0})}^{\f12}
\|\baruh_\la\|_{\wt{L}^2_{t,f}(\cB^{s,0})}^{\f12}\\
\leq &\|\uh_0\|_{\cB^{s,0}}
+\f12\|\na_\h\baruh_\la\|_{\wt L^2_t(\cB^{s,0})}
+C\|\baruh_\la\|_{\wt{L}^2_{t,f}(\cB^{s,0})}.
\end{split}
\eeno
By taking $\la={C^2}$ in the above inequality gives rise to
$$
\|\baruh_\la\|_{\wt{L}^\infty_t(\cB^{s,0})}+\|\na \baruh_\la\|_{\wt{L}^2_t(\cB^{s,0})}
\leq 2\|\uh_0\|_{\cB^{s,0}} \quad\mbox{for}\ \forall\ s>-1.
$$
Yet it follows from the definition \eqref{4.2} that
\beno
\begin{split}
\bigl(\|\baruh\|_{\wt{L}^\infty_t(\cB^{s,0})}
+\|\na \baruh\|_{\wt{L}^2_t(\cB^{s,0})}\bigr)
\exp\Bigl(-\lam\int_0^tf(t')\,dt'\Bigr)
\leq \|\baruh_\la\|_{\wt{L}^\infty_t(\cB^{s,0})}
+\|\na \baruh_\la\|_{\wt{L}^2_t(\cB^{s,0})},
\end{split}
\eeno
which implies that for any $s>-1,$ there holds
\beno
\|\baruh\|_{\wt{L}^\infty_t(\cB^{s,0})}+\|\na \baruh\|_{\wt{L}^2_t(\cB^{s,0})}
\leq 2\|\uh_0\|_{\cB^{s,0}}
\exp\bigl(C\|\nablah\baruh\|_{L^2_t(L^\infty_\v(L^2_\h))}^2\bigr).
\eeno
Along with \eqref{S4eq1}, we deduce \eqref{S4eq2}. This completes the proof of Proposition \ref{S4prop1}.
\end{proof}

\begin{prop}\label{S4prop2}
{\sl Under the assumptions of Proposition \ref{S4prop1}, for any $t\in [0,T]$ and $s\in ]-1,1[,$
we have
\beq \label{S4eq2a}\|\pa_3\baruh\|_{\wt{L}^\infty_t(\cB^{s,0})}+\|\na\pa_3 \baruh\|_{\wt{L}^2_t(\cB^{s,0})}
\leq 2\|\pa_3\uh_0\|_{\cB^{s,0}}
\exp\bigl(C A_\d(u_0^\h)\bigr).
\eeq}
\end{prop}

\begin{proof}
We first get, by applying $\pa_3$ to the equation \eqref{eqtbarulam}, that
\begin{equation*}\label{eqtpazbarulam}
\pa_t \pa_3\baruh_\lam+\la f(t)\pa_3\baruh_\lam
+\baruh\cdot\nablah\pa_3\baruh_\lam+\pa_3\baruh_\lam\cdot\nablah\baruh
-\Delta \pa_3\baruh_\lam=-\nablah \pa_3\bar p_\lam.
\end{equation*}
Applying $\D_k^\h$ to the above equation and taking $L^2$ inner product
of the resulting equation with
$\D_k^\h\pa_3\baruh_\lam,$ we obtain
\begin{equation}\begin{split}\label{4.11}
\f12\f{d}{dt}\|\D_k^\h&\pa_3\baruh_\la(t)\|_{L^2}^2
+\la f(t)\|\D_k^\h\pa_3\baruh_\la(t)\|_{L^2}^2
+\|\D_k^\h\na \pa_3\baruh_\la\|_{L^2}^2\\
&=-\bigl(\D_k^\h(\baruh\cdot
\na_\h\pa_3 \baruh_\la) | \D_k^\h\pa_3\baruh_\la\bigr)_{L^2}
-\bigl(\D_k^\h(\pa_3\baruh_\lam\cdot
\na_\h\baruh) | \D_k^\h\pa_3\baruh_\lam\bigr)_{L^2}.
\end{split}\end{equation}
Applying Lemma \ref{S4lem1} yields that for any $s\in ]-1,2[$
\begin{equation*}\begin{split}\label{4.16}
\int_0^t\bigl|\bigl(\D_k^\h(\baruh\cdot
\na_\h\pa_3 \baruh_\la)
\big| \D_k^\h\pa_3\baruh_\lam\bigr)_{L^2}\bigr|\,dt'
\lesssim d_k^2 2^{-2ks}\|\na_\h\pa_3\baruh_\la\|_{\wt{L}^2_t(\cB^{s,0})}
\|\pa_3\baruh_\la\|_{\wt{L}^2_{t,f}(\cB^{s,0})}.
\end{split}\end{equation*}
Applying Lemma \ref{S4lem2} gives for any   $s\in ]-1,1[$ that
\begin{equation*}\begin{split}\label{4.20}
\int_0^t\bigl|\bigl(\D_k^\h(\pa_3\baruh_\lam\cdot\na_\h\baruh)
\big| \D_k^\h\pa_3\baruh_\lam\bigr)_{L^2}\bigr|\,dt'
\lesssim d_k^2 2^{-2ks}\|\na_\h\pa_3\baruh_\la\|_{\wt{L}^2_t(\cB^{s,0})}
\|\pa_3\baruh_\la\|_{\wt{L}^2_{t,f}(\cB^{s,0})}.
\end{split}\end{equation*}
Integrating \eqref{4.11} over $[0,t]$ and inserting the above
two estimates into the resulting inequality, we achieve for any $ s\in]-1,1[$ that
 \beno
\begin{split}
\|\D_k^\h\pa_3\baruh_\la\|_{L^\infty_t(L^2)}^2&+\la \int_0^tf(t')\|\D_k^\h\pa_3\baruh_\la(t')\|_{L^2}^2\,dt'+\|\D_k^\h\na \pa_3\baruh\|_{L^2_t(L^2)}^2\\
&\leq \|\D_k^\h\pa_3\baruh_0\|_{L^2}^2+C d_k^2 2^{-2ks}
\|\na_\h\pa_3\baruh_\la\|_{\wt{L}^2_t(\cB^{s,0})}
\|\pa_3\baruh_\la\|_{\wt{L}^2_{t,f}(\cB^{s,0})}.
\end{split}
\eeno
In view of Definition \ref{def2}, by multiplying $2^{2ks}$ to the above inequality and taking square root
of the resulting inequality, and then summing up the
resulting inequalities over $\Z,$ we find
\beno
\begin{split}
\|\pa_3\baruh_\la\|_{\wt{L}^\infty_t(\cB^{s,0})}
+&\sqrt{\la}\|\pa_3\baruh_\la\|_{\wt{L}^2_{t,f}(\cB^{s,0})}
+\|\na \pa_3\baruh_\la \|_{\wt{L}^2_t(\cB^{s,0})}\\
\leq&\|\pa_3\baruh_0\|_{\cB^{s,0}}
+C\|\na_\h\pa_3\baruh_\la\|_{\wt L^2_t(\cB^{s,0})}^{\f12}
\|\pa_3\baruh_\la\|_{\wt{L}^2_{t,f}(\cB^{s,0})}^{\f12}.
\end{split}
\eeno
Then  by repeating the last step of
the proof to Proposition \ref{S4prop1},  we deduce \eqref{S4eq2a}. This finishes the proof of the proposition.
\end{proof}

Now we are in a position to complete the proof of  Proposition \ref{propbaru}.

\begin{proof}[Proof of Proposition \ref{propbaru}] It follows from Theorem \ref{S4thm1} that \eqref{eqtbaru} has a unique solution $u\in\frak{E}$ given by \eqref{ping1}.
It remains to deal with the estimates \eqref{estimatebaru1} and  \eqref{estimatebaru2}.
Indeed for any $\sigma_1\in ]-1,1[$ and $\sigma_2\in ]0,1[$,
and for any integer $N$, we decompose  the vertical frequency of $\baruh$ into the low and high frequency parts so that
\begin{align*}
\|\baruh\|_{\wt{L}^\infty_t(B^{\sigma_1,\sigma_2}_{2,1})}
& =\sum_{k,\ell\in\Z^2}2^{k\sigma_1}2^{\ell\sigma_2}
\|\D_k^\h\D_\ell^\v \baruh\|_{L^\infty_t(L^2)}\\
 &\lesssim\sum_{\ell\leq N}2^{\ell\sigma_2}
\sum_{k\in \Z}2^{k\sigma_1}\|\D_k^\h\baruh\|_{L^\infty_t(L^2)}
+\sum_{\ell> N}2^{\ell(\sigma_2-1)}
\sum_{k\in \Z}2^{k\sigma_1}\|\D_k^\h\pa_3\baruh\|_{L^\infty_t(L^2)}\\
&\lesssim 2^{N\sigma_2}\|\baruh\|_{\wt{L}^\infty_t(\cB^{\sigma_1,0})}
+2^{N(\sigma_2-1)}\|\pa_3\baruh\|_{\wt{L}^\infty_t(\cB^{\sigma_1,0})},
\end{align*} where we used Definition \ref{def2} in the last step.
Taking the integer $N$ so that
$ 2^N \sim \frac{\|\pa_3\baruh\|_{\wt{L}^\infty_t(\cB^{\sigma_1,0})}}
{\|\baruh\|_{\wt{L}^\infty_t(\cB^{\sigma_1,0})}} $
gives rise to
\begin{equation*}\label{4.22}
\|\baruh\|_{\wt{L}^\infty_t(B^{\sigma_1,\sigma_2}_{2,1})}
\lesssim \|\baruh\|_{\wt{L}^\infty_t(\cB^{\sigma_1,0})}^{1-\sigma_2}
\|\pa_3\baruh\|_{\wt{L}^\infty_t(\cB^{\sigma_1,0})}^{\sigma_2}.
\end{equation*}
Similarly,  we have
\begin{equation*}
\|\nabla\baruh\|_{\wt{L}^2_t(B^{\sigma_1,\sigma_2}_{2,1})}
\lesssim  \|\baruh\|_{\wt{L}^2_t(\cB^{\sigma_1,0})}^{1-\sigma_2}
\|\pa_3\baruh\|_{\wt{L}^2_t(\cB^{\sigma_1,0})}^{\sigma_2}.
\end{equation*}
Inserting \eqref{S4eq2} and \eqref{S4eq2a} into the above inequalities leads to
\begin{equation}\label{4.23}
\|\baruh\|_{\wt{L}^\infty_t(B^{\sigma_1,\sigma_2}_{2,1})}
+\|\nabla\baruh\|_{\wt{L}^2_t(B^{\sigma_1,\sigma_2}_{2,1})}
\lesssim \|\baruh_0\|_{\cB^{\sigma_1,0}}^{1-\sigma_2}
\|\pa_3\baruh_0\|_{\cB^{\sigma_1,0}}^{\sigma_2}
\exp\bigl(A_\d(\uh_0)\bigr).
\end{equation}
Then \eqref{estimatebaru1} follows from
\eqref{4.23} and Minkowski's inequality.

On the other hand, along the same line to the proof of Lemmas \ref{S4lem1} and \ref{S4lem2}, we get
\begin{equation*}\begin{split}
&\int_0^t\bigl|\bigl(\D_k^\h(\baruh\cdot
\na_\h\pa_3 \baruh_\la)
\big| \D_k^\h\pa_3\baruh_\lam\bigr)_{L^2}\bigr|\,dt'
\lesssim c_k^2 2^{-2k\s_1}\|\na_\h\pa_3\baruh_\la\|_{\wt{L}^2_t(H^{\s_1,0})}
\|\pa_3\baruh_\la\|_{\wt{L}^2_{t,f}(H^{\s_1,0})},\\
&\int_0^t\bigl|\bigl(\D_k^\h(\pa_3\baruh_\lam\cdot\na_\h\baruh)
\big| \D_k^\h\pa_3\baruh_\lam\bigr)_{L^2}\bigr|\,dt'
\lesssim c_k^2 2^{-2k\s_1}\|\na_\h\pa_3\baruh_\la\|_{\wt{L}^2_t(H^{\s_1,0})}
\|\pa_3\baruh_\la\|_{\wt{L}^2_{t,f}(H^{\s_1,0})},
\end{split}\end{equation*} where $\left(c_k\right)_{k\in\Z}$ is a generic element of $\ell^2(\Z)$ so that $\sum_{k\in\Z}c_k^2=1.$ Then by
integrating \eqref{4.11} over $[0,t]$ and inserting the above inequalities into the resulting inequality, we find
 \beno
\begin{split}
\|\D_k^\h\pa_3&\baruh_\la\|_{L^\infty_t(L^2)}^2+\la \int_0^tf(t')\|\D_k^\h\pa_3\baruh_\la(t')\|_{L^2}^2\,dt'+\|\D_k^\h\na \pa_3\baruh\|_{L^2_t(L^2)}^2\\
&\leq  \|\D_k^\h\pa_3\baruh_0\|_{L^\infty_t(L^2)}^2+Cc_k^2 2^{-2k\s_1}\|\na_\h\pa_3\baruh_\la\|_{\wt{L}^2_t(H^{\s_1,0})}
\|\pa_3\baruh_\la\|_{\wt{L}^2_{t,f}(H^{\s_1,0})}.
\end{split}
\eeno
Multiplying $2^{2k\s_1}$ to the above inequality and then
summing up the resulting inequalities over $\Z$ gives rise to
\beno
\begin{split}
\|\pa_3\baruh_\la\|_{\wt{L}^\infty_t(H^{\s_1,0})}^2
+&\la\|\pa_3\baruh_\la\|_{\wt{L}^2_{t,f}(H^{\s_1,0})}^2
+\|\na \pa_3\baruh_\la \|_{{L}^2_t(H^{\s_1,0})}^2\\
\leq&\|\pa_3\baruh_0\|_{H^{\s_1,0}}^2
+C\|\na_\h\pa_3\baruh_\la\|_{ L^2_t(H^{\s_1,0})}
\|\pa_3\baruh_\la\|_{\wt{L}^2_{t,f}(H^{\s_1,0})},
\end{split}
\eeno
from which, we deduce \eqref{estimatebaru2} by repeating the proof to the last step of Proposition \ref{S4prop1}.
\end{proof}

\subsection{The proof of Proposition \ref{propwtu}}

\begin{prop}\label{S4prop3}
{\sl Let $\wt u$ be a smooth enough solution of \eqref{eqtwtu}. Then under the assumptions of Theorem \ref{S4thm1}, for any $s_1\in ]-1,2[$
and  $s_2\in ]-1,1[,$
we have
\ben
 &&\|\wt u\|_{\wt{L}^\infty_t(\cB^{s_1,0})}+\|\na \wt u\|_{\wt{L}^2_t(\cB^{s_1,0})}
\leq C\|u_0\|_{\cB^{s_1,0}}
\exp\bigl(C A_\d(u_0^\h)\bigr), \label{S4eq9}\\
&&\|\pa_3\wt u\|_{\wt{L}^\infty_t(\cB^{s_2,0})}+\|\na\pa_3 \wt u\|_{\wt{L}^2_t(\cB^{s_2,0})}
\leq C\|\pa_3u_0\|_{\cB^{s_2,0}}\exp\bigl(C A_\d(u_0^\h)\bigr), \label{S4eq10}
\een where $A_\d (\uh_0)$ is determined by \eqref{S1eq3}.}
\end{prop}

\begin{proof} In view  of \eqref{eqtwtu}, we get,
by a similar derivation of \eqref{4.4} and \eqref{4.11},  that
\begin{equation}\label{4.24}
\f12\f{d}{dt}\|\D_k^\h\wtu_\la(t)\|_{L^2}^2
+\la f(t)\|\D_k^\h\wtu_\la(t)\|_{L^2}^2
+\|\D_k^\h\na \wtu_\la\|_{L^2}^2=-\bigl(\D_k^\h(\baruh\cdot
\na_\h \wtu_\la) \big| \D_k^\h\wtu_\la\bigr)_{L^2},
\end{equation}
and
\begin{equation}\begin{split}\label{4.25}
\f12\f{d}{dt}\|\D_k^\h&\pa_3\wtu_\la(t)\|_{L^2}^2
+\la f(t)\|\D_k^\h\pa_3\wtu_\la(t)\|_{L^2}^2
+\|\D_k^\h\na \pa_3\wtu_\la\|_{L^2}^2\\
&=-\bigl(\D_k^\h(\baruh\cdot
\na_\h\pa_3 \wtu_\la) | \D_k^\h\pa_3\wtu_\la\bigr)_{L^2}
-\bigl(\D_k^\h(\pa_3\baruh\cdot
\na_\h\wtu_\lam) | \D_k^\h\pa_3\wtu_\lam\bigr)_{L^2}.
\end{split}\end{equation}
Applying Lemma \ref{S4lem1} gives
\begin{equation*}\begin{split}\label{4.26}
\int_0^t\bigl|\bigl(\D_k^\h(\baruh\cdot
\na_\h \wtu_\la)
\big| \D_k^\h\wtu_\lam\bigr)_{L^2}\bigr|\,dt'
\lesssim d_k^2 2^{-2k s_1}\|\na_\h\wtu_\la\|_{\wt{L}^2_t(\cB^{s_1,0})}
\|\wtu_\la\|_{\wt{L}^2_{t,f}(\cB^{s_1,0})},
\end{split}\end{equation*}
and
\begin{equation*}\begin{split}\label{4.27}
\int_0^t\bigl|\bigl(\D_k^\h(\baruh\cdot
\na_\h\pa_3 \wtu_\la)
\big| \D_k^\h\pa_3\wtu_\lam\bigr)_{L^2}\bigr|\,dt'
\lesssim d_k^2 2^{-2ks_2}\|\na_\h\pa_3\wtu_\la\|_{\wt{L}^2_t(\cB^{s_2,0})}
\|\pa_3\wtu_\la\|_{\wt{L}^2_{t,f}(\cB^{s_2,0})}.
\end{split}\end{equation*}
Applying Lemma \ref{S4lem2} yields
\beno
\int_0^t\bigl|\bigl(\D_k^\h(\pa_3\baruh\cdot
\na_\h\wtu_\lam) | \D_k^\h\pa_3\wtu_\lam\bigr)_{L^2}\bigr|\,dt'\lesssim d_k^2 2^{-2ks_2}\|\na_\h\pa_3\wtu_\la\|_{\wt{L}^2_t(\cB^{s_2,0})}
\|\pa_3\wtu_\la\|_{\wt{L}^2_{t,f}(\cB^{s_2,0})}.
\eeno
By integrating \eqref{4.24} and \eqref{4.25} over $[0,t]$ and inserting the above estimates into the resulting inequalities, we find
\beno
\begin{split}
\|\D_k^\h\wtu_\la\|_{L^\infty_t(L^2)}^2
+&\la \int_0^tf(t')\|\D_k^\h\wtu_\la(t)\|_{L^2}^2\,dt'
+\|\D_k^\h\na \wtu_\la\|_{L^2_t(L^2)}^2\\
\leq &\|\D_k^\h\wtu_0\|_{L^\infty_t(L^2)}^2+Cd_k^2 2^{-2k s_1}\|\na_\h\wtu_\la\|_{\wt{L}^2_t(\cB^{s_1,0})}
\|\wtu_\la\|_{\wt{L}^2_{t,f}(\cB^{s_1,0})},
\end{split}
\eeno
and
\begin{equation*}\begin{split}
\|\D_k^\h\pa_3\wtu_\la\|_{L^\infty_t(L^2)}^2
&+\la \int_0^tf(t')\|\D_k^\h\pa_3\wtu_\la(t')\|_{L^2}^2\,dt'
+\|\D_k^\h\na \pa_3\wtu_\la\|_{L^2_t(L^2)}^2\\
&\leq \|\D_k^\h\pa_3\wtu_0\|_{L^2}^2+Cd_k^2 2^{-2ks_2}\|\na_\h\pa_3\wtu_\la\|_{\wt{L}^2_t(\cB^{s_2,0})}
\|\pa_3\wtu_\la\|_{\wt{L}^2_{t,f}(\cB^{s_2,0})}.
\end{split}\end{equation*}
With the above inequalities, by taking $\la$ larger than a uniform constant, we can follow the same line as that used in the proof of Propositions \ref{S4prop1} to show that
$$\|\wtu\|_{\wt{L}^\infty_t(\cB^{s_1,0})}
+\|\na\wtu\|_{\wt{L}^2_t(\cB^{s_1,0})}
\leq 2\|\wtu_0\|_{\cB^{s_1,0}}\exp\bigl(C A_\d(u_0^\h)\bigr),$$
and
$$\|\pa_3\wtu\|_{\wt{L}^\infty_t(\cB^{s_2,0})}
+\|\na\pa_3\wtu\|_{\wt{L}^2_t(\cB^{s_2,0})}
\leq 2\|\pa_3\wtu_0\|_{\cB^{s_2,0}}\exp\bigl(C A_\d(u_0^\h)\bigr).$$
Then  the estimates \eqref{S4eq9} and \eqref{S4eq10} follow
once we notice that
$$\|\wtu_0\|_{\cB^{s_1,0}}=\bigl\|\nablah\Deltah^{-1}(\diveh\uh_0)
\bigr\|_{\cB^{s_1,0}}+\|u_0^3\|_{\cB^{s_1,0}}
\lesssim\|u_0\|_{\cB^{s_1,0}},$$
and similarly $\|\pa_3\wtu_0\|_{\cB^{s_2,0}}\lesssim\|\pa_3 u_0\|_{\cB^{s_2,0}}.$ This completes the proof of the proposition.
\end{proof}

Let us present the proof of Proposition \ref{propwtu}.

\begin{proof}[Proof of Proposition \ref{propwtu}]
We first get, by a similar derivation of \eqref{4.23}, that
\begin{equation*}
\begin{split}
\|\wtu\|_{{L}^\infty_t(B^{\s_1,\s_2}_{2,1})}
+\|\nabla\wtu\|_{{L}^2_t(B^{\s_1,\s_2}_{2,1})}
\lesssim &\|\wtu\|_{{L}^\infty_t(\cB^{\s_1,0})}^{1-\s_2}
\|\pa_3\wtu\|_{{L}^\infty_t(\cB^{\s_1,0})}^{\s_2}\\
&+\|\nabla\wtu\|_{{L}^2_t(\cB^{\s_1,0})}^{1-\s_2}
\|\nabla\pa_3\wtu\|_{{L}^2_t(\cB^{\s_1,0})}^{\s_2},
\end{split}
\end{equation*}
which together with  Proposition \ref{S4prop3} ensures  \eqref{S2eq1}.

Whereas along the same line to the proof of Lemmas \ref{S4lem1} and \ref{S4lem2}, we infer
\begin{equation*}\begin{split}
&\int_0^t\bigl|\bigl(\D_k^\h(\baruh\cdot
\na_\h\pa_3 \wtu_\la)
\big| \D_k^\h\pa_3\wtu_\lam\bigr)_{L^2}\bigr|\,dt'
\lesssim c_k^2 2^{-2k\s_1}\|\na_\h\pa_3\wtu_\la\|_{\wt{L}^2_t(H^{\s_1,0})}
\|\pa_3\wtu_\la\|_{\wt{L}^2_{t,f}(H^{\s_1,0})};\\
&\int_0^t\bigl|\bigl(\D_k^\h(\pa_3\baruh\cdot
\na_\h\wtu_\lam) | \D_k^\h\pa_3\wtu_\lam\bigr)_{L^2}\bigr|\,dt'\lesssim c_k^2 2^{-2k\s_1}\|\na_\h\pa_3\wtu_\la\|_{\wt{L}^2_t(H^{\s_1,0})}
\|\pa_3\wtu_\la\|_{\wt{L}^2_{t,f}(H^{\s_1,0})}.
\end{split}\end{equation*}
By integrating \eqref{4.25} over $[0,t]$ and inserting the above inequalities into the resulting inequality, we find
 \beno
\begin{split}
\|\D_k^\h\pa_3&\wt{u}_\la\|_{L^\infty_t(L^2)}^2+\la \int_0^tf(t')\|\D_k^\h\pa_3\wt{u}_\la(t')\|_{L^2}^2\,dt'
+\|\D_k^\h\na \pa_3\wt{u}_\lam\|_{L^2_t(L^2)}^2\\
&\lesssim  \|\D_k^\h\pa_3\wt{u}_0\|_{L^\infty_t(L^2)}^2+c_k^2 2^{-2k\s_1}\|\na_\h\pa_3\wt{u}_\la\|_{\wt{L}^2_t(H^{\s_1,0})}
\|\pa_3\wt{u}_\la\|_{\wt{L}^2_{t,f}(H^{\s_1,0})}.
\end{split}
\eeno
Then  along the same line to the derivation of \eqref{estimatebaru2},
we achieve \eqref{S2eq2}.

It remains to prove \eqref{S2eq3}. We first get, by taking $H^{\s_2,0}$ inner product
of the $\wt{u}^\h$ equation in \eqref{eqtwtu} with $\wt{u}^\h,$ that
\beq \label{4.32}
\f12\f{d}{dt}\|\wt{u}^\h(t)\|_{H^{\s_2,0}}^2+\|\na\wt{u}^\h(t)\|_{H^{\s_2,0}}^2=-\bigl(\bar{u}^\h\cdot\na_\h\wt{u}^\h | \wt{u}^\h\bigr)_{H^{\s_2,0}}-\bigl(\na_\h \wt{p} | \wt{u}^\h\bigr)_{H^{\s_2,0}}.
\eeq
Note that $\s_2\in]0,1[,$ applying the law of product, Lemma \ref{lemproductlaw}, yields
\beno
\begin{split}
\bigl|\bigl(\bar{u}^\h\cdot\na_\h\wt{u}^\h | \wt{u}^\h\bigr)_{H^{\s_2,0}}\bigr|\leq &\|\bar{u}^\h\cdot\na_\h\wt{u}^\h\|_{H^{\s_2,0}}\|\wt{u}^\h\|_{H^{\s_2,0}}\\
\leq &C\|\bar{u}^\h\|_{B^{1,\f12}_{2,1}}\|\na_\h\wt{u}^\h\|_{H^{\s_2,0}}\|\wt{u}^\h\|_{H^{\s_2,0}}\\
\leq &C\|\na_\h \bar{u}^\h\|_{B^{0,\f12}_{2,1}}^2\|\wt{u}^\h\|_{H^{\s_2,0}}^2+\f14\|\na_\h\wt{u}^\h\|_{H^{\s_2,0}}^2.
\end{split}
\eeno

To handle the pressure term, we shall use the decomposition
$\wt{p}=\wt{p}_1+\wt{p}_2$ with $\wt{p}_1$ and $\wt{p}_2$ being given by \eqref{S2eq10}.
Then due to $\s_2\in]0,1[,$ by applying Lemma \ref{lemproductlaw}, we deduce that
\beno
\begin{split}
\bigl|\bigl(\na_\h\wt{p}_1  | \wt{u}^\h\bigr)_{H^{\s_2,0}}\bigr|\leq & \|\wt{u}^\h\cdot\na_\h\bar{u}^\h\|_{H^{\s_2-1,0}}\|\wt{u}^\h\|_{H^{1+\s_2,0}}\\
\leq& C\|\wt{u}^\h\|_{H^{\s_2,0}}\|\na_\h\bar{u}^\h\|_{B^{0,\f12}_{2,1}}\|\na_\h\wt{u}^\h\|_{H^{\s_2,0}}\\
\leq& C\|\na_\h\bar{u}^\h\|_{B^{0,\f12}_{2,1}}^2\|\wt{u}^\h\|_{H^{\s_2,0}}^2+\f14\|\na_\h\wt{u}^\h\|_{H^{\s_2,0}}^2,
\end{split}
\eeno
and
\beno
\begin{split}
\bigl|\bigl(\na_\h\wt{p}_2  | \wt{u}^\h\bigr)_{H^{\s_2,0}}\bigr|\leq & \|\wt{u}^3\pa_3\bar{u}^\h\|_{H^{\s_2,0}}\|\wt{u}^\h\|_{H^{\s_2,0}}\\
\leq& C\|\wt{u}^3\|_{B^{1,\f12}_{2,1}}\|\pa_3\bar{u}^\h\|_{H^{\s_2,0}}\|\wt{u}^\h\|_{H^{\s_2,0}}\\
\leq& C\|\na_\h\wt{u}^3\|_{B^{0,\f12}_{2,1}}^2\|\wt{u}^\h\|_{H^{\s_2,0}}^2+\|\na_\h\pa_3\bar{u}^\h\|_{H^{\s_2-1,0}}^2.
\end{split}
\eeno
By inserting the above estimates into \eqref{4.32}, we achieve
\beno
\begin{split}
\f{d}{dt}\|\wt{u}^\h\|_{H^{\s_2,0}}^2+\|\na\wt{u}^\h\|_{H^{\s_2,0}}^2
\leq C\bigl(\|\na_\h\bar{u}^\h\|_{B^{0,\f12}_{2,1}}^2
+\|\na_\h\wt{u}^3\|_{B^{0,\f12}_{2,1}}^2\bigr)\|\wt{u}^\h\|_{H^{\s_2,0}}^2
+\|\na_\h\pa_3\bar{u}^\h\|_{H^{\s_2-1,0}}^2.
\end{split}
\eeno
Applying Gronwall's inequality yields
\beq \label{4.38}
\begin{split}
\|\wt{u}^\h\|_{L^\infty_t(H^{\s_2,0})}^2+\|\na\wt{u}^\h\|_{L^2_t(H^{\s_2,0})}^2
\leq \bigl(&\|\wt u^\h_0\|_{H^{\s_2,0}}^2
+\|\na_\h\pa_3\bar{u}^\h\|_{L^2_t(H^{\s_2-1,0})}^2\bigr)\\
&\times\exp\Bigl(C\bigl(\|\na_\h\bar{u}^\h\|_{L^2_t(B^{0,\f12}_{2,1})}^2
+\|\na_\h\wt{u}^3\|_{L^2_t(B^{0,\f12}_{2,1})}^2\bigr)\Bigr).
\end{split}
\eeq
Yet it follows from Proposition \ref{propbaru} that
\beno
\begin{split}
&\|\na_\h\bar{u}^\h\|_{L^2_t(B^{0,\f12}_{2,1})}\leq C\|u_0^\h\|_{\cB^{0,0}}^{\frac12}\|\pa_3u_0^\h\|_{\cB^{0,0}}^{\frac12}\exp\bigl(C A_\d(u_0^\h)\bigr),\\
&\|\na_\h\pa_3\bar{u}^\h\|_{L^2_t(H^{\s_2-1,0})}\leq C\|\pa_3u_0^\h\|_{H^{\s_2-1,0}}\exp\bigl(C A_\d(u_0^\h)\bigr).
\end{split}
\eeno
Whereas \eqref{S2eq1} implies that
\beno
\|\na_\h\wt{u}^3\|_{L^2_t(B^{0,\f12}_{2,1})}\leq C\|u_0^\h\|_{\cB^{0,0}}^{\frac12}\|\pa_3u_0^\h\|_{\cB^{0,0}}^{\frac12}\exp\bigl(C A_\d(u_0^\h)\bigr).
\eeno
By inserting the above inequalities and the fact that
$\|\wt u^\h_0\|_{H^{\s_2,0}}\leq\|\pa_3u_0^3\|_{H^{\s_2-1,0}}$ into \eqref{4.38}, we achieve
 \eqref{S2eq3}.
This completes the proof of Proposition \ref{propwtu}.
\end{proof}

\setcounter{equation}{0}
\section{{\it A priori} estimates of the System \eqref{eqtv3om}}\label{sec5}

\subsection{The Proof of the estimate \eqref{ineqv3}}

We first observe that
$$\|\nabla v^3\|_{H^{-\f12,0}}^2
=\|v^3\|_{H^{\f12,0}}^2+\|v^3\|_{H^{-\f12,1}}^2.$$
Then the estimates of $\|\nabla v^3\|_{H^{-\f12,0}}$ is reduced to handle the estimate of
$\|v^3\|_{H^{\f12,0}}$ and $\|v^3\|_{H^{-\f12,1}}.$

By taking the $\hh$ and $\hmo$ inner product
of the $v^3$ equation in \eqref{eqtv3om} with $v^3$ respectively, we get
\begin{equation}\begin{split}\label{5.1}
\f12\f{d}{dt}\|&v^3\|_{\hh\cap\hmo}^2+\nvn_{\hh\cap\hmo}^2
=-\bigl(Q\,\big|\,v^3\bigr)_{\hh\cap\hmo},\end{split}\end{equation}
where $Q=\sum_{i=1}^6Q_i$ and
\begin{equation}\begin{split}\label{5.2a}
&Q_1=v\cdot\nabla v^3,\qquad
Q_2=v\cdot\nabla \wt u^3+\wt u\cdot\nabla v^3
+\baruh\cdot\nablah v^3,\\
& Q_3=\wt u\cdot\nabla \wt u^3, \qquad Q_4=\pD\sum_{\ell,m=1}^3\pa_\ell v^m\pa_m v^\ell,\\
&Q_5=2\pD\sum_{\ell,m=1}^3\pa_\ell v^m\pa_m (\baru^\ell+\wt u^\ell),\\
&Q_6=\pD\sum_{\ell,m=1}^3\Bigl(\pa_\ell (\baru^m+\wt u^m)
\pa_m (\baru^\ell+\wt u^\ell)-\pa_\ell \baru^m \pa_m\wt u^\ell\Bigr).
\end{split}\end{equation}
Here and in the rest of this section, we always denote
\beno
\|\cdot\|_{X\cap Y}\eqdefa \|\cdot\|_X+\|\cdot\|_Y \andf \left( f | g \right)_{X\cap Y}\eqdefa \left( f | g \right)_{X}+\left( f | g \right)_{ Y}.
\eeno

Next let us handle  term by term in \eqref{5.2a}

\noindent$\bullet$\underline{
The estimate of $\bigl(Q_{1}\,\big|\, v^3\bigr)
_{H^{\frac12,0}\cap\hmo}$.}

Applying H\"older's inequality and  the law of product, Lemma \ref{lemproductlaw}, we obtain
\begin{equation}\begin{split}\label{5.2}
\bigl|\bigl(Q_{1}\,\big|\, v^3\bigr)_{H^{\frac12,0}\cap\hmo}\bigr|
&\leq\|v\cdot\nabla v^3\|_{H^{-\f12,0}}\|v^3\|_{H^{\f32,0}\cap H^{-\f12,2}}\\
&\lesssim \|v\|_{H^{\f12,\f14}}\|\nabla v^3\|_{H^{0,\f14}}
\|v^3\|_{H^{\f32,0}\cap H^{-\f12,2}}.
\end{split}\end{equation}
While it follows from \eqref{Helmholtz} that
\begin{equation*}\begin{split}
\|v\|_{H^{\f12,\f14}}
&\lesssim \|\nablah\vh\|_{H^{-\f12,\f14}}+\|v^3\|_{H^{\f12,\f14}}\\
&\lesssim \|\omega\|_{H^{-\f12,\f14}}+\|\pa_3v^3\|_{H^{-\f12,\f14}}
+\|\nablah v^3\|_{H^{-\f12,\f14}}.
\end{split}\end{equation*}
Note that
\beno
\begin{split}
\|a\|_{H^{-\f12,\f14}}^2=&\int_{\R^3}\left(|\xi_\h|^{-1}\right)^{\f34}\left(|\xi_\h|^{-1}|\xi_3|^2\right)^{\f14}|\hat{a}(\xi)|^2\,d\xi\\
\leq &\Bigl(\int_{\R^3}|\xi_\h|^{-1}|\hat{a}(\xi)|^2\,d\xi\Bigr)^{\f34}\Bigl(\int_{\R^3}|\xi_\h|^{-1}|\xi_3|^2|\hat{a}(\xi)|^2\,d\xi\Bigr)^{\f14}=\|a\|_{H^{-\f12,0}}^{\f32}
\|\pa_3a\|_{H^{-\f12,0}}^{\f12},
\end{split}
\eeno
 and recalling
 the definition of $M(t)$ and $N(t)$ given by \eqref{defMN}, we infer
\begin{equation}\begin{split}\label{5.3}
\|v(t)\|_{H^{\f12,\f14}}
\lesssim M^{\f38}(t)N^{\f18}(t).
\end{split}\end{equation}
On the other hand, it is easy to observe that
\beq\label{5.23a}
\begin{split}
\|a\|_{H^{0,\f14}}^2=&\int_{\R^3}\left(|\xi_\h|^{-1}\right)^{\f14}\left(|\xi_\h|^{-1}|\xi_\h|^{\f43}|\xi_3|^{\f23}\right)^{\f34}|\hat{a}(\xi)|^2\,d\xi\\
\leq &\Bigl(\int_{\R^3}|\xi_\h|^{-1}|\hat{a}(\xi)|^2\,d\xi\Bigr)^{\f14}\Bigl(\int_{\R^3}|\xi_\h|^{-1}|\xi|^2|\hat{a}(\xi)|^2\,d\xi\Bigr)^{\f34}\\
=&
\|a\|_{H^{-\f12,0}}^{\f12}\|\na a\|_{H^{-\f12,0}}^{\f32},
\end{split}
\eeq
which together with \eqref{defMN}  implies
\begin{equation*}
\|\nabla v^3(t)\|_{H^{0,\f14}}
\leq M^{\f18}(t)N^{\f38}(t).
\end{equation*}
Inserting the above estimates and $
\|v^3(t)\|_{H^{\f32,0}\cap H^{-\f12,2}}\leq N^{\f12}(t)$  into \eqref{5.2}, we deduce that
\begin{equation}\label{5.5}
\bigl|\bigl(Q_{1}\,\big|\, v^3\bigr)_{H^{\frac12,0}\cap\hmo}\bigr|
\leq CM^{\f12}(t)N(t).
\end{equation}

\noindent$\bullet$\underline{
The estimate of $\bigl(Q_{2}+Q_3\,\big|\, v^3\bigr)
_{H^{\frac12,0}\cap\hmo}$.}

By applying H\"older's inequality and Lemma \ref{lemproductlaw}, we get
\begin{equation*}\begin{split}
\bigl|\bigl(Q_{2}\,\big|\, v^3\bigr)_{H^{\frac12,0}\cap\hmo}\bigr|
&\leq\|v\cdot\nabla \wt u^3+(\wt u+\baru)\cdot\nabla v^3\|_{H^{-\f12,0}}
\|v^3\|_{H^{\f32,0}\cap H^{-\f12,2}}\\
&\lesssim\bigl(\|v\|_{H^{\f12,0}}\|\nabla\wt u^3\|_\bh
+\|\wt u+\baru\|_\boh
\|\nabla v^3\|_{H^{-\f12,0}}\bigr)
 N^{\f12}(t).
\end{split}\end{equation*}
In view of \eqref{Helmholtz} and \eqref{defMN}, we infer
\begin{equation}\begin{split}\label{5.7}
\|v(t)\|_\hh&\lesssim \|\nablah\vh(t)\|_{H^{-\f12,0}}+\|v^3(t)\|_{H^{\f12,0}}\\
&\lesssim\|\omega(t)\|_{H^{-\f12,0}}+\|\pa_3v^3(t)\|_{H^{-\f12,0}}+
\|v^3(t)\|_{H^{\f12,0}}\leq M^{\f12}(t).
\end{split}\end{equation}
This in turn shows that
\begin{equation}\begin{split}\label{5.9}
\bigl|\bigl(Q_{2}\,\big|\, v^3\bigr)_{H^{\frac12,0}\cap\hmo}\bigr|
&\leq C M^{\f12}(t)N^{\f12}(t)\bigl(\|\nabla\wt u\|_\bh
+\|\nabla\baru^\h\|_\bh\bigr)\\
&\leq\f{1}{100} N(t)+CM(t)\bigl(\|\nabla\wt u\|_\bh^2
+\|\nabla\baru^\h\|_\bh^2\bigr).
\end{split}\end{equation}

Along the same line, we have
\begin{equation*}\begin{split}\label{5.10}
\bigl|\bigl(Q_3\,\big|\, v^3\bigr)_{H^{\frac12,0}\cap\hmo}\bigr|
&\leq\bigl(\|\wt u^\h\cdot\nablah\wt u^3\|_\hm
+\|\wt u^3\pa_3\wt u^3\|_\hm\bigr)\|v^3\|_{H^{\f32,0}\cap H^{-\f12,2}}\\
&\lesssim\bigl(\|\wt u^\h\|_{H^{\f12,0}}\|\nablah\wt u^3\|_\bh
+\|\wt u^3\|_\bh\|\pa_3\wt u^3\|_{H^{\f12,0}}\bigr) N^{\f12}(t).
\end{split}\end{equation*}
Applying Young's inequality yields
\begin{equation}\label{5.11}
\bigl|\bigl(Q_3\,\big|\, v^3\bigr)_{H^{\frac12,0}\cap\hmo}\bigr|
\leq\f{1}{100} N(t)+C\bigl(\|\wt u^\h\|_{H^{\f12,0}}^2\|\nabla\wt u^3\|_\bh^2
+\|\wt u^3\|_\bh^2\|\nabla\pa_3\wt u^3\|_{H^{-\f12,0}}^2\bigr).
\end{equation}

\noindent$\bullet$\underline{
The estimate of $\bigl(Q_4\,\big|\, v^3\bigr)
_{H^{\frac12,0}\cap\hmo}$.}

Let us first handle the estimate of $\bigl(Q_4\,\big|\, v^3\bigr)
_{H^{\frac12,0}}.$ Indeed it is easy to observe that
\begin{equation}\begin{split}\label{5.12}
\bigl|\bigl(Q_4\,\big|\, v^3\bigr)_{H^{\frac12,0}}\bigr|
\leq&\Bigl\|\pD\Bigl(\sum_{\ell,m=1}^2\pa_\ell v^m\pa_m v^\ell
+(\pa_3 v^3)^2\Bigr)\Bigr\|_\hh\|v^3\|_\hh\\
&
+2\bigl\|\pD ( \pa_3 v^\h\cdot\na_\h v^3)
\bigr\|
_{H^{\f12,-\f14}}\|v^3\|_{H^{\f12,\f14}}\\
\leq &\Bigl\|\sum_{\ell,m=1}^2\pa_\ell v^m\pa_m v^\ell
+(\pa_3 v^3)^2\Bigr\|_\hm\|v^3\|_\hh\\
&+2\bigl\|\pa_3 v^\h\cdot\na_\h v^3\bigr\|
_{H^{-\f12,-\f14}}\|v^3\|_{H^{\f12,\f14}}.
\end{split}\end{equation}
It follows from the law of product, Lemma \ref{lemproductlaw}, and
 \eqref{Helmholtz} that
\begin{equation*}\begin{split}\label{5.13}
\Bigl\|\sum_{\ell,m=1}^2\pa_\ell v^m\pa_m v^\ell
+(\pa_3 v^3)^2\Bigr\|_\hm
&\lesssim\|\nablah\vh\|_{H^{\f14,\f14}}^2+\|\pa_3 v^3\|_{H^{\f14,\f14}}^2\\
&\lesssim\|\omega\|_{H^{\f14,\f14}}^2+\|\pa_3 v^3\|_{H^{\f14,\f14}}^2,
\end{split}\end{equation*}
and
\begin{equation*}\begin{split}\label{5.14}
\bigl\|\pa_3 v^\h\cdot\na_\h v^3\bigr\|_{H^{-\f12,-\f14}}
&\lesssim\|\pa_3\vh\|_\hh\|\nablah v^3\|_{H^{0,\f14}}\\
&\lesssim\bigl(\|\pa_3\omega\|_\hm+\|\pa_3^2 v^3\|_\hm\bigr)\|\nablah v^3\|_{H^{0,\f14}}.
\end{split}\end{equation*}
Note that
\beq\label{5.12a}
\begin{split}
\|a\|_{H^{\f14,\f14}}^2=&\int_{\R^3}|\xi_\h|^{\f34}\left(|\xi_\h|^{-1}|\xi_3|^2\right)^{\f14}|\hat{a}(\xi)|^2\,d\xi\\
\leq &\Bigl(\int_{\R^3}|\xi_\h||\hat{a}(\xi)|^2\,d\xi\Bigr)^{\f34}\Bigl(\int_{\R^3}|\xi_\h|^{-1}|\xi_3|^2|\hat{a}(\xi)|^2\,d\xi\Bigr)^{\f14}\\
=& \|a\|_{H^{\f12,0}}^{\f32}
\|\pa_3 a\|_{H^{-\f12,0}}^{\f12},
\end{split}
\eeq
and
\beno
\begin{split}
\|\na_\h a\|_{H^{0,\f14}}^2\leq &\int_{\R^3}|\xi_\h|^{\f14}\left(|\xi_\h||\xi_3|^2\right)^{\f34}|\hat{a}(\xi)|^2\,d\xi\\
\leq &\Bigl(\int_{\R^3}|\xi_\h||\hat{a}(\xi)|^2\,d\xi\Bigr)^{\f14}\Bigl(\int_{\R^3}|\xi_\h||\xi|^2|\hat{a}(\xi)|^2\,d\xi\Bigr)^{\f34}=\|a\|_{H^{\f12,0}}^{\f12}
\|\na a\|_{H^{\f12,0}}^{\f32},
\end{split}
\eeno
we find
\begin{equation}\begin{split}\label{5.13}
\Bigl\|\sum_{\ell,m=1}^2\pa_\ell v^m\pa_m v^\ell
+(\pa_3 v^3)^2\Bigr\|_\hm
\lesssim N(t),
\end{split}\end{equation}
and
\begin{equation}\begin{split}\label{5.14}
\bigl\|\pa_3 v^\h\cdot\na_\h v^3 \bigr\|_{H^{-\f12,-\f14}}
\lesssim M^{\f18}(t)N^{\f78}(t).
\end{split}\end{equation}

Inserting the estimates \eqref{5.3}, \eqref{5.13} and \eqref{5.14} into
\eqref{5.12} leads to
\begin{equation}\label{5.15}
\bigl|\bigl(Q_4\,\big|\, v^3\bigr)_{H^{\frac12,0}}\bigr|
\leq C M^{\f12}(t)N(t).
\end{equation}

Let us turn to the estimate of  $\bigl(Q_4\,\big|\, v^3\bigr)_{H^{-\frac12,1}}$. Observing that
\begin{equation*}\begin{split}\label{5.16}
\bigl|\bigl(Q_4\,\big|\, v^3\bigr)_{H^{-\frac12,1}}\bigr|
\leq &\Bigl\|\pD\Bigl(\sum_{\ell,m=1}^2\pa_\ell v^m\pa_m v^\ell
+(\pa_3 v^3)^2\Bigr)\Bigr\|_\hmo\|v^3\|_\hmo\\
&
+2\bigl\|\pD(\pa_3 v^\h\cdot\na_\h v^3)\bigr\|_{H^{-\f12,\f34}}\|v^3\|_{H^{-\f12,\f54}}\\
\leq&\Bigl\|\sum_{\ell,m=1}^2\pa_\ell v^m\pa_m v^\ell
+(\pa_3 v^3)^2\Bigr\|_\hm\|\pa_3v^3\|_\hm\\
&+2\bigl\|\pa_3 v^\h\cdot\na_\h v^3\bigr\|_{H^{-\f12,-\f14}}\|v^3\|_{H^{-\f12,\f54}},
\end{split}\end{equation*}
from which, and \eqref{5.13}, \eqref{5.14}, and the fact that
$$\|v^3\|_{H^{-\f12,\f54}}\leq\|\pa_3v^3\|_{H^{-\f12,0}}^{\f34}
\|\pa_3^2 v^3\|_{H^{-\f12,0}}^{\f14}\leq M^{\f38}(t)N(t)^{\f18},$$
we deduce
\begin{equation}\label{5.17}
\bigl|\bigl(Q_4\,\big|\, v^3\bigr)_{\hmo}\bigr|
\leq C M^{\f12}(t)N(t).
\end{equation}

Combining  \eqref{5.15} with \eqref{5.17}, we achieve
\begin{equation}\label{5.18}
\bigl|\bigl(Q_4\,\big|\, v^3\bigr)_{H^{\frac12,0}\cap\hmo}\bigr|
\leq C M^{\f12}(t)N(t).
\end{equation}

\noindent$\bullet$\underline{
The estimate of $\bigl(Q_5+Q_6\,\big|\, v^3\bigr)
_{H^{\frac12,0}\cap\hmo}$.}

Due to $\dive\bar{u}=\dive\wt{u}=0,$ we  write
$$Q_5=2\sum_{\ell,m=1}^3\pD\pa_\ell
\bigl( v^m\pa_m (\baru^\ell+\wt u^\ell)\bigr).$$
Applying the law of product, Lemma \ref{lemproductlaw},  yields
\begin{equation*}\begin{split}
\bigl|\bigl(Q_{5}\,\big|\, v^3\bigr)_{H^{\frac12,0}\cap\hmo}\bigr|
&\lesssim\Bigl\|\sum_{\ell,m=1}^3\pD\pa_\ell
\bigl( v^m\pa_m (\baru^\ell+\wt u^\ell)\bigr)\Bigr\|_{H^{-\f12,0}}
\|v^3\|_{H^{\f32,0}\cap H^{-\f12,2}}\\
&\lesssim\|v\|_{H^{\f12,0}}\bigl(\|\nabla\bar u\|_\bh
+\|\nabla \wt u\|_\bh\bigr)
 N^{\f12}(t),
\end{split}\end{equation*}
from which and \eqref{5.7}, we deduce
\begin{equation}\label{5.20}
\bigl|\bigl(Q_{5}\,\big|\, v^3\bigr)_{H^{\frac12,0}\cap\hmo}\bigr|
\leq \f{1}{100} N(t)+CM(t)\bigl(\|\nabla\bar u^\h\|_\bh^2
+\|\nabla \wt u\|_\bh^2\bigr).
\end{equation}

Similarly, by using $\dive\bar{u}=\dive\wt{u}=0$ once again, we write
$$
Q_6=2\sum_{\ell,m=1}^3\pa_\ell\pa_m\Delta^{-1}
\Bigl(\pa_3(\baru^m+\wt u^m)(\baru^\ell+\wt u^\ell)
-\pa_3\baru^m\wt u^\ell\Bigr).
$$
It follows from the law of product, Lemma \ref{lemproductlaw}, that
\begin{equation*}\begin{split}\label{5.21}
&\Bigl\|\sum_{\ell,m=1}^3\pa_\ell\pa_m\Delta^{-1}
\Bigl(\pa_3(\baru^m+\wt u^m)(\baru^\ell+\wt u^\ell)
-\pa_3\baru^m\wt u^\ell
\Bigr)\Bigr\|_{H^{-\f12,0}}\|v^3\|_{H^{\f32,0}\cap H^{-\f12,2}}\\
&\lesssim \bigl(\|\pa_3\baru\|_\hh+\|\pa_3\wt u\|_\hh\bigr)
\bigl(\|\baru\|_\bh+\|\wt u\|_\bh\bigr) N^{\f12}(t).
\end{split}\end{equation*}
Then by applying convexity inequality, we find
\begin{equation}\label{5.22}
\bigl|\bigl(Q_{6}\,\big|\, v^3\bigr)_{H^{\frac12,0}\cap\hmo}\bigr|
\leq \f{1}{100} N(t)+C\bigl(\|\pa_3\baru\|_\hh^2+\|\pa_3\wt u\|_\hh^2\bigr)
\bigl(\|\baru\|_\bh^2+\|\wt u\|_\bh^2\bigr).
\end{equation}

By inserting the estimates \eqref{5.5}-\eqref{5.11},
and \eqref{5.18}-\eqref{5.22} into \eqref{5.1}, we achieve
 \eqref{ineqv3}.

\subsection{The Proof of the estimate \eqref{ineqom}}

By taking the $\hm$ inner product of the $\omega$ equation in \eqref{eqtv3om}
with $\omega$, we get
\begin{equation}\begin{split}\label{5.23}
&\f12\f{d}{dt}\|\omega\|_{\hm}^2+\nomn_{\hm}^2
=-\bigl(I_1+\cdots+I_6\,\big|\,\omega\bigr)_\hm,\quad\mbox{with}\\
&I_1=v\cdot\nabla\omega-\om\pa_3 v^3,\qquad
I_2=\wt u^\h\cdot\nablah(\barom+\wt\om)-(\barom+\wt\om)\pa_3 u^3,\\
&I_3=\wt u^3\pa_3(\barom+\wt\om),\qquad\quad
\ I_4=v\cdot\nabla(\barom+\wt\om)+(\baru+\wt u)\cdot\nabla\omega
-\omega\pa_3\wt u^3,\\
&I_5=\pa_3 u^\h\cdot\nablah^\perp u^3,\qquad\quad
\ I_6=-\pa_2\baruh\cdot\nablah\wt u^1
+\pa_1\baruh\cdot\nablah\wt u^2.
\end{split}\end{equation}

Next we handle  term by term above.

\noindent$\bullet$\underline{
The estimate of $\bigl(I_{1}\,\big|\, \om\bigr)_\hm$.}

Due to $\dive v=0,$ we get, by using integration by parts, that
$$\bigl(v\cdot\nabla\omega\,\big|\,\omega\bigr)_\hm
=\bigl(\dive(v\omega)\,\big|\,\omega\bigr)_\hm
=-\bigl(v\omega\,\big|\,\nabla\omega\bigr)_\hm.$$
Then it follows from  Lemma \ref{lemproductlaw}
that
\begin{equation*}\begin{split}
\bigl|\bigl(I_{1}\,\big|\, \om\bigr)_{\hm}\bigr|
&\leq\|v\om\|_{H^{-\f12,0}}\|\nabla\om\|_{\hm}
+\|\om\pa_3 v^3\|_{H^{-\f12,0}}\|\om\|_{\hm}\\
&\leq\|v\|_{H^{\f12,\f14}}\|\om\|_{H^{0,\f14}}\|\nabla\om\|_{\hm}
+\|\om\|_{H^{\f14,\f14}}\|\pa_3 v^3\|_{H^{\f14,\f14}}\|\om\|_{\hm},
\end{split}\end{equation*}
from which,  \eqref{5.3},  \eqref{5.23a} and \eqref{5.12a}, we deduce that
\begin{equation}\begin{split}\label{5.24}
\bigl|\bigl(I_{1}\,\big|\, v^3\bigr)_{\hm}\bigr|
&\lesssim M^{\f38}(t)N^{\f18}(t)\cdot M^{\f18}(t)N^{\f38}(t)\cdot N^{\f12}(t)+N^{\f12}(t)\cdot N^{\f12}(t)\cdot M^{\f12}(t)\\
&\lesssim M^{\f12}(t)N(t).
\end{split}\end{equation}

\noindent$\bullet$\underline{
The estimate of $\bigl(I_{2}\,\big|\, \om\bigr)_\hm$.}

Due to $\dive \wt{u}=0,$ we have $\dive_\h\wt{u}^\h=-\pa_3\wt{u}$ and
$$\bigl(\wt u^\h\cdot\nablah(\barom+\wt\om)\,\big|\, \om\bigr)_{\hm}
=\bigl(\diveh\bigl((\barom+\wt\om)\wt u^\h\bigr)
+(\barom+\wt\om)\pa_3\wt u^3\,\big|\, \om\bigr)_{\hm}\,,$$
so that we get, by using integration by parts and the law of product, Lemma \ref{lemproductlaw}, that
\begin{equation*}\begin{split}\label{5.25}
\bigl|\bigl(\diveh\bigl(\wt u^\h(\barom+\wt\om)\bigr)
\,\big|\, \om\bigr)_{\hm}\bigr|
&\leq\|\wt u^\h(\barom+\wt\om)\|_\hm\|\nablah\om\|_\hm\\
&\leq C\|\wt u^\h\|_\hh\|\barom+\wt\om\|_\bh\|\nabla\om\|_\hm\\
&\leq\f{1}{100}\|\nabla\om\|_\hm^2+C\|\wt u^\h\|_\hh^2
\bigl(\|\barom\|_\bh^2+\|\wt\om\|_\bh^2\bigr).
\end{split}\end{equation*}
Applying  the law of product, Lemma \ref{lemproductlaw}, once again yields
\begin{equation*}\begin{split}\label{5.26}
\bigl|\bigl((\barom+\wt\om)\pa_3\wt u^3\,\big|\, \om\bigr)_{\hm}\bigr|
&\leq\|(\barom+\wt\om)\pa_3\wt u^3\|_\hm\|\om\|_\hm\\
&\leq C\|\pa_3\wt u^3\|_\hh\|(\barom+\wt\om)\|_\bh\|\om\|_\hm\\
&\leq C\|\om\|_\hm^2\|(\barom+\wt\om)\|_\bh^2
+\f1{100}\|\pa_3\wt u^3\|_\hh^2.
\end{split}\end{equation*}
And exactly along the same line, we find
\begin{equation*}\label{5.27}
\bigl|\bigl((\barom+\wt\om)\pa_3 u^3\,\big|\, \om\bigr)_{\hm}\bigr|
\leq C\|\om\|_\hm^2\|\barom+\wt\om\|_\bh^2
+\f1{100}\|\pa_3 (v^3+\wt u^3)\|_\hh^2.
\end{equation*}
As a result, it comes out
\begin{equation}\label{5.28}
\bigl|\bigl(I_2\,\big|\, \om\bigr)_{\hm}\bigr|
\leq C M(t)\bigl(\|\nabla\baru^\h\|_\bh^2+\|\nabla\wt u\|_\bh^2\bigr)
+C\|\nabla_\h\pa_3\wt u^3\|_\hm^2+\f1{100}N(t).
\end{equation}

\noindent$\bullet$\underline{
The estimate of $\bigl(I_{3}\,\big|\, \om\bigr)_\hm $.}

We deduce from the law of product, Lemma \ref{lemproductlaw}, that
\begin{equation}\begin{split}\label{5.29}
\bigl|\bigl(I_3\,\big|\, \om\bigr)_{\hm}\bigr|
&\leq\|\wt u^3\pa_3(\barom+\wt\om)\|_\hm\|\om\|_\hm\\
&\lesssim\|\wt u^3\|_\boh\|\pa_3(\barom+\wt\om)\|_\hm\|\om\|_\hm\\
&\lesssim\|\om\|_\hm^2 \|\nablah\wt u^3\|_\bh^2
+\bigl(\|\pa_3\barom\|_\hm^2+\|\pa_3\wt\om\|_\hm^2\bigr)\\
&\lesssim M(t) \|\nabla\wt u\|_\bh^2+\bigl(\|\nabla\pa_3\baru^\h\|_\hm^2
+\|\nabla\pa_3\wt u\|_\hm^2\bigr).
\end{split}\end{equation}

\noindent$\bullet$\underline{
The estimate of $\bigl(I_{4}\,\big|\, \om\bigr)_\hm $.}

Due to $\dive v=0,$ we get,  by using integration by parts and then
Lemma \ref{lemproductlaw}, that
\begin{equation*}\begin{split}\label{5.30}
\bigl|\bigl(v\cdot\nabla(\barom+\wt\om)\,\big|\,\om\bigr)_{\hm}\bigr|
&=\bigl|\bigl(v(\barom+\wt\om)\,\big|\,\nabla \om\bigr)_{\hm}\bigr|\\
&\leq\|v(\barom+\wt\om)\|_\hm\|\nabla\om\|_\hm\\
&\leq C\|v\|_\hh\|\barom+\wt\om\|_\bh\|\nabla\om\|_\hm\\
&\leq \f1{100}\|\nabla\om\|_\hm^2+C\|v\|_\hh^2
\bigl(\|\nabla_\h\baru^\h\|_\bh^2
+\|\nabla_\h\wt u\|_\bh^2\bigr),
\end{split}\end{equation*}
and
\begin{equation*}\begin{split}\label{5.31}
\bigl|\bigl((&\baru+\wt u)\cdot\nabla\omega
 -\omega\pa_3\wt u^3\,\big|\,\om\bigr)_{\hm}\bigr|\\
&\leq\|(\baru+\wt u)\cdot\nabla\omega
-\omega\pa_3\wt u^3\|_\hm\|\om\|_\hm\\
&\leq C\bigl(\|\baru+\wt u\|_\boh\|\nabla\om\|_\hm
+\|\om\|_\hh\|\pa_3\wt u^3\|_\bh\bigr)\|\om\|_\hm\\
&\leq \f1{100}\|\nabla\om\|_\hm^2+C\|\om\|_\hm^2
\bigl(\|\nabla\baru^\h\|_\bh^2
+\|\nabla\wt u\|_\bh^2\bigr).
\end{split}\end{equation*}
By summing up the above two estimates, we achieve
\begin{equation}\label{5.32}
\bigl|\bigl(I_{4}\,\big|\, \om\bigr)_\hm\bigr|
\leq CM(t)
\bigl(\|\nabla\baru^\h\|_\bh^2
+\|\nabla\wt u\|_\bh^2\bigr)+\f1{50}N(t).
\end{equation}

\noindent$\bullet$\underline{
The estimate of $\bigl(I_{5}\,\big|\, \om\bigr)_\hm$.}

As $u^3=\wt{u}^3+v^3,$ We shall deal with the estimate involving $\wt{u}^3$ and $\v^3$ separately. We first get,
by applying the law of product, Lemma \ref{lemproductlaw}, that
\begin{equation}\begin{split}\label{5.33}
\bigl|\bigl(\pa_3 u^\h\cdot\nablah^\perp v^3\,\big|\, \om\bigr)_{\hm}\bigr|
&\leq\|\pa_3 u^\h\cdot\nablah^\perp v^3\|_{H^{-\f12,-\f14}}
\|\om\|_{H^{-\f12,\f14}}\\
&\lesssim\|\pa_3 u^\h\|_\hh\|\nablah^\perp v^3\|_{H^{0,\f14}}
\|\om\|_{H^{-\f12,\f14}}.
\end{split}\end{equation}
In view of  \eqref{Helmholtz}, we have
\begin{equation}\begin{split}\label{5.34}
\|\pa_3 u^\h\|_\hh
&\leq \|\nablah\pa_3 v^\h\|_\hm
+\|\nablah\pa_3 (\baru^\h+\wt u^\h)\|_\hm\\
&\leq \|\pa_3 \omega\|_\hm+\|\pa_3^2 v^3\|_\hm
+\|\nabla\pa_3( \baru^\h+\wt u^\h)\|_\hm\\
&\leq N^{\f12}(t)
+\|\nabla\pa_3( \baru^\h+\wt u^\h)\|_\hm.
\end{split}\end{equation}
Whereas it follows from \eqref{5.23a} that
\begin{equation*}\begin{split}\label{5.35}
&\|\nablah^\perp v^3\|_{H^{0,\f14}}\leq \|v^3\|_\hh^{\f14}
\|\nabla v^3\|_\hh^{\f34}\leq M^{\f18}(t)N^{\f38}(t),\quad\mbox{and}\\
&\|\om\|_{H^{-\f12,\f14}}\leq\|\om\|_{H^{-\f12,0}}^{\f34}
\|\pa_3\om\|_{H^{-\f12,0}}^{\f14}\leq M^{\f38}(t)N^{\f18}(t).
\end{split}\end{equation*}
Inserting the above estimates into \eqref{5.33}
and using convexity inequality gives rise to
\begin{equation*}\begin{split}\label{5.36}
\bigl|\bigl(\pa_3 u^\h\cdot\nablah^\perp v^3\,\big|\, \om\bigr)_{\hm}\bigr|
\leq CM^{\f12}(t)N(t)+M(t) N(t)
+C\|\nabla\pa_3( \baru^\h+\wt u^\h)\|_\hm^2.
\end{split}\end{equation*}
Along the same line, by using \eqref{5.34}, we obtain
\begin{equation*}\begin{split}\label{5.37}
\bigl|\bigl(\pa_3 u^\h\cdot\nablah^\perp \wt u^3\,\big|\, \om\bigr)_{\hm}\bigr|
&\leq\|\pa_3 u^\h\cdot\nablah^\perp \wt u^3\|_\hm\|\om\|_\hm\\
&\leq C\|\pa_3\uh\|_\hh\|\nablah^\perp \wt u^3\|_\bh\|\om\|_\hm\\
&\leq C\bigl(N^{\f12}(t)
+\|\nabla\pa_3( \baru^\h+\wt u^\h)\|_\hm\bigr)
\|\nabla\wt u\|_\bh M^{\f12}(t)\\
&\leq CM(t)\|\nabla\wt u\|_\bh^2
+C\|\nabla\pa_3( \baru^\h+\wt u^\h)\|_\hm^2+\f1{100}N(t).
\end{split}\end{equation*}
By combining the above two estimates,
we achieve
\begin{equation}\label{5.38}\begin{split}
\bigl|\bigl(I_5\,\big|\, \om\bigr)_{\hm}\bigr|
\leq &\bigl(\f1{100}+CM^{\f12}(t)+M(t)\bigr)N(t)\\
&+CM(t)\|\nabla\wt u\|_\bh^2+C\|\nabla\pa_3( \baru^\h+\wt u^\h)\|_\hm^2.
\end{split}\end{equation}

\noindent$\bullet$\underline{
The estimate of $\bigl(I_{6}\,\big|\, \om\bigr)_\hm$.}

 We get,
by applying the law of product, Lemma \ref{lemproductlaw}, that
\begin{equation}\begin{split}\label{5.39}
\bigl|\bigl(I_6\,\big|\, \om\bigr)_{\hm}\bigr|
&\leq\|-\pa_2\baruh\cdot\nablah\wt u^1
+\pa_1\baruh\cdot\nablah\wt u^2\|_\hm\|\om\|_\hm\\
&\lesssim\|\nablah \baru^\h\|_\bh\|\nablah\wt u^\h\|_\hh\|\om\|_\hm\\
&\lesssim M(t)\|\nabla \baru^\h\|_\bh^2
+\|\nabla\wt u^\h\|_\hh^2.
\end{split}\end{equation}

Substituting the estimates \eqref{5.24}-\eqref{5.32},~\eqref{5.38} and \eqref{5.39} into \eqref{5.23}, we obtain \eqref{ineqom}.
This concludes the proof of Proposition \ref{aprioriv}.

\setcounter{equation}{0}
\section{The proof of Theorem \ref{thmmain}}\label{sec6}

The purpose of this section is to complete the proof of Theorem \ref{thmmain}.
The strategy is to verify that, the necessary condition, \eqref{blowupCZ5},
for the finite time blow-up of Fujita-Kato solutions to $(NS)$ can never be satisfied.
Let us first present the proof of Proposition \ref{S6prop1}.

\begin{proof}[Proof of Proposition \ref{S6prop1}]
As explained in Section \ref{sec2}, let us consider the maximal solution
$v\in \cE_{T^\ast}$
of \eqref{eqtv}, where $T^\ast$ denotes the lifespan of $(NS)$ in \eqref{FTS}.
Let $M(t), N(t)$ be determined by \eqref{defMN},  we denote
\begin{equation}\label{S6eq1}
T^\star \eqdef\sup\Bigl\{\, T\in ]0,T^\ast[\ :
\sup_{t\in [0,T[}\bigl(M(t)+\int_0^t N(t')\,dt'\bigr)\leq \eta
\, \Bigr\}
\end{equation}
for some sufficiently small constant $\eta$, which will be determined later on.

We are going to prove that $T^\star=T^\ast$ for some particular choice of $\eta.$ Otherwise,
if $T^\star<T^\ast,$  for any $t\in[0,T^\star]$,
we get, by summing up the estimates \eqref{ineqv3},~\eqref{ineqom}
and using the bound in \eqref{S6eq1}, that
\begin{equation} \label{S6eq1a} \begin{split}
\f{dM(t)}{dt}+2N(t)
\leq &\bigl(\f12+C\eta^{\f12}+\eta\bigr)N
+CM\bigl(\|\nabla\wt u\|_\bh^2+\|\nabla\baru^\h\|_\bh^2\bigr)\\
&+C\|\nabla\wt u^\h\|_\hh^2+C\|\wt u^\h\|_{H^{\f12,0}}^2\|\nabla\wt u^3\|_\bh^2\\
&+C\bigl(1+\|\baruh\|_\bh^2+\|\wt u\|_\bh^2\bigr)\bigl(\|\nabla\pa_3\baruh\|_{H^{-\f12,0}}^2
+\|\nabla\pa_3\wt u\|_{H^{-\f12,0}}^2\bigr).
\end{split}\end{equation}
In particular if we take $\eta=\min\bigl(\f14, \f1{16C^2}\bigr)$,  the
first term on the righthand side of \eqref{S6eq1a} can be absorbed by $2N(t)$ on its left-hand side.
Then by applying Gronwall's inequality to the resulting inequality, we achieve
\beq \label{S6eq2}
\begin{split}
M(t)&+\int_0^tN(t')\,dt'
\leq C\Bigl\{\|\nabla\wt u^\h\|_{L^2_t(\hh)}^2
+\|\wt u^\h\|_{L^\infty_t(H^{\f12,0})}^2\|\nabla\wt u^3\|_{L^2_t(\bh)}^2\\
&+\bigl(1+\|(\baruh,\wt u)\|_{L^\infty_t(\bh)}^2\bigr)
\|\nabla\pa_3(\baruh,\wt u)\|_{L^2_t(H^{-\f12,0})}^2\Bigr\}
\exp\bigl(C\|\nabla(\baru^\h,\wt u)\|_{L^2_t(\bh)}^2\bigr).
\end{split}
\eeq
On the other hand, it follows from Propositions \ref{propbaru} and
\ref{propwtu} and \eqref{S3eq6} that
\begin{equation}\label{S6eq5}
\begin{split}
\|(&\baruh,\wt u)\|_{{L}^\infty_t(B^{0,\f12}_{2,1})}^2
+\|\nabla(\baruh,\wt u)\|_{{L}^2_t(B^{0,\f12}_{2,1})}^2\\
&\lesssim\|\uh_0\|_{\cB^{0,0}}
\|\pa_3\uh_0\|_{\cB^{0,0}}
\exp\bigl(CA_\d(\uh_0)\bigr)\\
&\lesssim\|u_0\|_{H^{-\d,0}}^{\f12}\|u_0\|_{H^{\d,0}}^{\f12}
\|\pa_3u_0\|_{H^{-\frac12,0}}^{\f12}
\|\pa_3u_0\|_{H^{\frac12,0}}^{\f12}
\exp\bigl(CA_\d(\uh_0)\bigr)\eqdefa B_\d(u_0),
\end{split}
\end{equation}
and
\begin{equation*}
\|\nabla\pa_3(\baruh,\wt u)\|_{{L}^2_t(H^{-\f12,0})}^2
\lesssim\|\pa_3u_0\|_{H^{-\f12,0}}^2
\exp\bigl(CA_\d(\uh_0)\bigr).
\end{equation*}
Whereas it follows from \eqref{S2eq3} and \eqref{S3eq6} that
\beno
\begin{split}
\|\wt{u}^\h\|_{L^\infty_t(H^{\f12,0})}^2+\|\na\wt{u}^\h\|_{L^2_t(H^{\f12,0})}^2
\lesssim\|\pa_3u_0\|_{H^{-\f12,0}}^2\exp\bigl(CB_\d(u_0)\bigr).
\end{split}
\eeno
Inserting the above estimates into \eqref{S6eq2} gives rise to
\beq \label{S6eq3}
M(t)+\int_0^tN(t')\,dt'
\leq \|\pa_3u_0\|_{H^{-\f12,0}}^2
\exp\bigl(CA_\d(\uh_0)+CB_\d(u_0)\bigr).
\eeq
In particular, if $\e_0$ in \eqref{S1eq1} is so small that $\e_0\leq \f{\eta}2,$
then we can deduce from \eqref{S6eq3} that
\beno
\sup_{t\in [0,T^\star]}\Bigl(M(t)+\int_0^tN(t')\,dt'\Bigr)
\leq \f\eta2,
\eeno
which contradicts with the choice of $T^\star$ given in \eqref{S6eq1}.
This in turn shows that $T^\star=T^\ast.$
\end{proof}

Let us now turn to the proof of Theorem \ref{thmmain}.

\begin{proof}[Proof of Theorem \ref{thmmain}] It is easy to observe from Lemma \ref{lemBern} that
\beno
\begin{split}
\|\D_j a\|_{L^\infty}\lesssim& \sum_{\substack{k\leq j+N_0\\ \ell\leq j+N_0}}
2^k2^{\f{\ell}2}\|\D_k^\h\D_\ell^\v a\|_{L^2}\\
\lesssim& \sum_{\substack{k\leq j+N_0\\ \ell\leq j+N_0}}
d_{k,\ell}2^k\|a\|_{B^{0,\f12}_{2,1}}\lesssim d_j 2^j\|a\|_{B^{0,\f12}_{2,1}},
\end{split}
\eeno
which implies that
$$\|a\|_{B^{-1}_{\infty,\infty}}\lesssim \|a\|_{B^{0,\f12}_{2,1}}.$$
Thanks to the above inequality and \eqref{S6eq5}, we have
\beq\label{S6eq7}
\int_0^{T^\ast}\|\na(\baruh,\wt u)\|_{B^{-1}_{\infty,\infty}}^2\,dt
\lesssim \int_0^{T^\ast}\|\na(\baruh,\wt u)\|_{B^{0,\f12}_{2,1}}^2\,dt
\lesssim B_\d(u_0).
\eeq
On the other hand, by virtue of \eqref{Helmholtz} and Lemma \ref{lemBern}, we infer
\beno
\begin{split}
\|\D_j(\na v^\h)\|_{L^\infty}\lesssim &2^{\f{j}2}\sum_{k\leq j+N_0}
2^k\|\D_k^\h(\na v^\h)\|_{L^2}\\
\lesssim &2^{\f{j}2}\sum_{k\leq j+N_0}
\bigl(\|\D_k^\h(\na \om)\|_{L^2}+\|\D_k^\h(\na \pa_3v^3)\|_{L^2}\bigr)\\
\lesssim &2^{\f{j}2}\sum_{k\leq j+N_0}
c_k(t)2^{\f{k}2}\bigl(\|\na\om\|_{H^{-\f12,0}}+\|\na \pa_3v^3\|_{H^{-\f12,0}}\bigr)\\
\lesssim &c_j(t)2^j\bigl(\|\na\om\|_{H^{-\f12,0}}+\|\na \pa_3v^3\|_{H^{-\f12,0}}\bigr),
\end{split}
\eeno where $\left(c_j(t)\right)_{j\in\Z}$ denotes a generic element of $\ell^2(\Z)$ so that $\sum_{j\in\Z}c_j^2(t)=1.$
This together with \eqref{S6eq4} ensures that
\beq\label{S6eq8}
\int_0^{T^\ast}\|\na v^\h(t)\|_{B^{-1}_{\infty,\infty}}^2\,dt\lesssim \int_0^{T^\ast}\bigl(\|\na\om\|_{H^{-\f12,0}}^2
+\|\na \pa_3v^3\|_{H^{-\f12,0}}^2\bigr)\,dt\leq \eta.
\eeq
Similarly, by applying Lemma \ref{lemBern}, we get
\beno
\begin{split}
\|\D_j(\na v^3)\|_{L^\infty}\lesssim & 2^{-j}\|\D_j(\na^2 v^3)\|_{L^\infty}
\lesssim 2^{-\f{j}2}\sum_{k\leq j+N_0}2^k\|\D_k^\h\na^2v^3\|_{L^2}\\
\lesssim &2^{-\f{j}2}\sum_{k\leq j+N_0}
c_k(t)2^{\f{3k}2}\|\na^2v^3(t)\|_{H^{-\f12,0}}\\
\lesssim& c_j(t)2^j\|\na^2v^3(t)\|_{H^{-\f12,0}},
\end{split}
\eeno
from which and \eqref{S6eq4}, we deduce that
\beno
\int_0^{T^\ast}\|\na v^3(t)\|_{B^{-1}_{\infty,\infty}}^2\,dt\lesssim \|\na^2v^3\|_{L^2_t(H^{-\f12,0})}^2\leq \eta.
\eeno
This together with \eqref{S6eq7} and \eqref{S6eq8} permits us to conclude that
\beno
\int_0^{T^\ast}\|\na u(t)\|_{B^{-1}_{\infty,\infty}}^2\,dt<\infty.
\eeno
Then Theorem \ref{blowupBesovendpoint} ensures that $T^\ast=\infty.$
This completes the proof of Theorem \ref{thmmain}.
\end{proof}

\bigbreak \noindent {\bf Acknowledgments.} P. Zhang would like to thank Professor
J.-Y. Chemin for profitable discussions on this topic.
P. Zhang is partially supported
by NSF of China under Grants   11371347 and 11688101, Morningside Center of Mathematics of Chinese Academy of Sciences and innovation grant from National Center for
Mathematics and Interdisciplinary Sciences.


\medskip

\begin{thebibliography}{50}

\bibitem{Bo81} J.-M. Bony, Calcul symbolique et propagation des
   singularit\'es pour les \'equations aux d\'eriv\'ees partielles non
   lin\'eaires, {\it  Ann. Sci. \'Ec. Norm. Sup\'er.},  {\bf
   14} (1981), 209-246.

%\bibitem{bg} H.    Bahouri and I.  Gallagher, On the stability in weak topology of the set of global solutions to the Navier-Stokes equations,
% {\it Arch. Ration. Mech. Anal.}, {\bf  209 } (2013),  569-629.

\bibitem{bcdbookk}
 H. Bahouri, J.-Y. Chemin and R. Danchin,
{\it Fourier Analysis and Nonlinear Partial Differential Equations}, Grundlehren der mathematischen Wissenschaften, {\bf 343}, Springer-Verlag Berlin Heidelberg, 2011.

%\bibitem{bcg2} H. Bahouri, J.-Y. Chemin and I. Gallagher,  Stability by rescaled weak convergence for the Navier-Stokes equations,
% {\it C. R. Math. Acad. Sci. Paris}, {\bf  352}  (2014),   305-310.

\bibitem{BP08} J. Bourgain and N. Pavlovi\'c,
Ill-posedness of the Navier-Stokes equations in a critical space in 3D,
{\it J. Funct. Anal.}, {\bf 255} (2008), 2233-2247.


\bibitem{cgens} J.-Y. Chemin and I. Gallagher, On  the global wellposedness  of the 3-D
   Navier-Stokes equations with large initial data, {\it  Ann. Sci. \'Ec. Norm. Sup\'er.}, {\bf 39} (2006),
 679--698.

\bibitem{CG10} J.-Y. Chemin and I. Gallagher,
Large, global solutions to the Navier-Stokes equations, slowly varying in one direction, {\it Trans. Amer. Math. Soc.},
 {\bf 362} (2010),  2859-2873.

\bibitem{CGZ} J.-Y. Chemin, I. Gallagher and P. Zhang,
Sums of large global solutions to the incompressible Navier-Stokes equations,
{\it J. Reine Angew. Math.}, {\bf681} (2013), 65-82.

\bibitem{CL} J.-Y. Chemin and N. Lerner,  Flot de champs de vecteurs non
lipschitziens et \'equations de Navier-Stokes, {\it J. Differential
Equations}, {\bf 121} (1995), 314-328.

\bibitem {CZ1}
J.-Y. Chemin and P.  Zhang, On the global wellposedness  to the 3-D incompressible anisotropic
 Navier-Stokes equations, {\it Comm.  Math. Phys.}, {\bf 272} (2007),  529-566.

\bibitem{CZ15} J.-Y. Chemin and P.  Zhang,
 Remarks on the global solutions of 3-D Navier-Stokes system with one slow variable, {\it Comm. Partial Differential Equations}, {\bf40} (2015),  878-896.

\bibitem{CZ5}
J.-Y. Chemin and P.  Zhang, On the critical one component regularity
for 3-D Navier-Stokes system. {\it  Ann. Sci. \'Ec. Norm. Sup\'er.}, {\bf 49} (2016), 131-167.

 \bibitem{cannonemeyerplanchon}
M. Cannone, Y. Meyer and F. Planchon,
Solutions autosimilaires des {\'e}quations de Navier-Stokes,
{S{\'e}minaire ``{\'E}quations aux D{\'e}riv{\'e}es Partielles" de l'{\'E}cole polytechnique},
Expos{\'e} VIII, 1993--1994.

%\bibitem{gipcras} I.  Gallagher, D. Iftimie and F. Planchon,  Asymptotics and stability for global solutions to the Navier-Stokes equations,
 %{\it Ann. Inst. Fourier} (Grenoble),   {\bf 53}  (2003),   1387-1424.

\bibitem{iftimieraugelsell} D. Iftimie, G. Raugel and G.~R. Sell,
Navier-Stokes equations in thin 3D domains with Navier boundary conditions, {\em  Indiana Univ. Math. J. }, {\bf 56}  (2007),  1083--1156.

\bibitem{fujitakato}
H. Fujita and T. Kato, On the Navier-Stokes initial value problem I.
{\it Arch. Ration. Mech. Anal.}, {\bf 16} (1964), 269-315.

\bibitem{kato} T. Kato,
Strong $L^p$-solutions of the Navier-Stokes equation in
$\R^m$ with applications to weak solutions, {\it  Math. Z.},
 {\bf 187} (1984),   471-480.

 \bibitem{KZ2} I. Kukavica and M.  Ziane, Navier-Stokes equations with
regularity in one direction, {\it  J. Math. Phys.},
{\bf 48} (2007), 065-203.

\bibitem {kochtataru}
H. Koch and D. Tataru, Well-posedness for
the Navier-Stokes equations, {\em  Adv. Math.}, {\bf 157} (2001),  22-35.

\bibitem{La} O.~A. Ladyzhenskaya,
Unique global solvability of the three-dimensional Cauchy problem for the Navier-Stokes
 equations in the presence of axial symmetry, (Russian) {\it Zap. Nau$\breve{c}$n. Sem. Leningrad. Otdel. Mat. Inst. Steklov. (LOMI)},
 {\bf 7} (1968), 155-177.

\bibitem {lerayns}
J. Leray, Essai sur le mouvement d'un liquide visqueux emplissant
l'espace, {\em Acta Math.}, {\bf 63} (1933),  193--248.

\bibitem{meyer} Y. Meyer.
{\sl Wavelets, Paraproducts and Navier--Stokes}.
Current Developments in Mathematics, International Press, Cambridge,
Massachussets, 1996.

\bibitem{Pa02} M. Paicu, \'Equation anisotrope
de Navier-Stokes dans des espaces  critiques,
{\it  Rev. Mat. Iberoam.},  {\bf 21} (2005), 179-235.

\bibitem{PZ1} M. Paicu and P. Zhang, Global solutions to the 3-D incompressible
 anisotropic Navier-Stokes system  in the critical spaces,
  {\it Comm.  Math. Phys.}, {\bf 307} (2011), 713-759.



\bibitem{raugelsell}
G. Raugel and G.~R. Sell,  Navier-Stokes equations on thin $3$D domains. I.
 Global attractors and global regularity of solutions,
{\it J. Amer. Math. Soc.},~{\bf 6} (1993),
 503--568.

 \bibitem{UY} M.~ R. Ukhovskii, and V. I. Iudovich,  Axially symmetric flows of ideal and viscous fluids
filling the whole space, {\it J. Appl. Math. Mech.}, {\bf 32} (1968)
52-61.
\end{thebibliography}
\end{document}